\documentclass[letterpaper,11pt]{article}

\usepackage{url}            
\usepackage{amsfonts}
\usepackage{latexsym}
\usepackage{epsfig}
\usepackage{graphicx}
\usepackage{amsfonts}       
\usepackage{nicefrac}       
\usepackage{microtype}      
\usepackage{amsmath, amssymb,amsthm}
\usepackage{algorithm,algorithmic}
\usepackage{color}
\usepackage{mathtools, eqparbox}
\usepackage{makecell}
\usepackage{hyperref}       

\newcommand{\tsum}{\textstyle{\sum}}
\newcommand{\beq}{\begin{equation}}
\newcommand{\eeq}{\end{equation}}
\newcommand{\nn}{\nonumber}
\newcommand{\bbe}{\mathbb{E}}
\newcommand{\bbr}{\mathbb{R}}
\def\la{\langle}
\def\ra{\rangle}
\def\1b{\mathbf{1}}
\def\0b{\mathbf{0}}
\def\Lb{\mathbf{L}}
\def\zb{\mathbf{z}}
\def\xb{\mathbf{x}}
\def\yb{\mathbf{y}}
\def\ub{\mathbf{u}}
\def\Gc{\mathcal{G}}
\def\Ec{\mathcal{E}}
\def\Vc{\mathcal{N}}  
\def\vb{\mathbf{s}}   
\def\o{\omega}
\def\vgap{\vspace*{.1in}}
\DeclareMathOperator*{\argmin}{arg\,min}

\def\inner{{\rm ACS}}

\def\algone{{\rm ADPD}}
\def\algtwo{{\rm AA-SDCS}}
\newcommand{\varstackrel}[3][T]{\stackrel{\raisebox{0.5ex}{\clap{\scriptsize#2}}}{#3}}

\newtheorem{theorem}{Theorem}
\newtheorem{lemma}[theorem]{Lemma}

\newtheorem{proposition}[theorem]{Proposition}

\newtheorem{assumption}{Assumption}
\newtheorem{definition}{Definition}

\title{Asynchronous decentralized accelerated stochastic gradient descent}

\date{}
\author{
  Guanghui Lan\and 
  Yi Zhou\thanks{Department of Industrial and Systems Engineering, Georgia Institute of Technology, Atlanta, GA, 30332. 
  (email:{\tt george.lan@isye.gatech.edu, yizhou@gatech.edu})}
}

\begin{document}

\maketitle

\begin{abstract}
  In this work, we introduce an asynchronous decentralized accelerated stochastic gradient descent type of method for decentralized stochastic optimization, considering communication and synchronization are the major bottlenecks. 
  We establish $\mathcal{O}(1/\epsilon)$ (resp., $\mathcal{O}(1/\sqrt{\epsilon})$) communication complexity and $\mathcal{O}(1/\epsilon^2)$ (resp., $\mathcal{O}(1/\epsilon)$) sampling complexity for solving general convex (resp., strongly convex) problems.
\end{abstract}

\section{Introduction}
In this paper, we consider the following decentralized optimization problem which is cooperatively solved by $m$ agents distributed over the network:
\begin{align}\label{eqn:orgprob}
f^*:=\min_x &~f(x) := \tsum_{i=1}^m f_i(x)\\
\text{s.t. }&~ x \in X, \quad X := \cap_{i=1}^m X_i \nonumber.
\end{align}
Here $f_i:X_i \to \mathbb{R}$ is a general convex objective function only known to agent $i$ and satisfying
\begin{align}\label{eqn:nonsmooth}
\tfrac{\mu}{2}\|x-y\|^2\le 
f_i(x)-f_i(y)-\langle f_i'(y), x-y\rangle \le \tfrac{L}{2}\|x-y\|^2 + M\|x-y\|, \ \ \forall x, y\in X_i,
\end{align}
for some $L, M,\mu\ge 0$ and $f_i'(y) \in \partial f_i(y)$, where $\partial f_i(y)$ denotes the subdifferential of $f_i$ at $y$, and $X_i \subseteq \mathbb{R}^d$ is a closed convex constraint set of agent $i$. 
\eqref{eqn:nonsmooth} is a unified way of describing a wide range of problems. In particular, if $f_i$ is a general Lipschitz continuous function with constant $M_f$, then \eqref{eqn:nonsmooth} holds with $L = 0, \mu = 0$ and $M = 2M_f$. If $f_i$ is a smooth and strongly convex function in $\mathcal{C}^{1,1}_{L/\mu}$ (see \cite[Section 1.2.2]{Nest04} for definition), \eqref{eqn:nonsmooth} is satisfied with $M = 0$. Clearly, relation \eqref{eqn:nonsmooth} also holds if
$f_i$ is given as the summation of smooth and nonsmooth convex functions. Throughout the paper, we assume the feasible set $X$ is nonempty.

Decentralized optimization problems defined over complex multi-agent networks are ubiquitous in
signal processing, machine learning, control, and other areas in science and engineering
(see e.g. \cite{rabbat,con01,ram_info,Durham-Bullo}). 
One critical issue existing in decentralized optimization is that synchrony among network agents is usually inefficient or impractical due to processing and communication delays and the absence of a master server in the network.
Note that $f_i$ and $X_i$ are private and only known to agent $i$,
and all agents intend to cooperatively minimize the system objective $f$ as the sum of all local objective $f_i$'s in the absence of full knowledge about the global problem and network structure. Decentralized algorithms, therefore, require agents to communicate with their neighboring agents iteratively to propagate the distributed information in the network. Under the synchronous setting, all agents must wait for the slowest agent and/or slowest communication channel/edge in the network, and a global coordinator must be presented for synchronization, which can be extremely expensive in the large-scale decentralized network. 

Following the seminal work \cite{Bertsekas:1997}, extensive research work has been conducted in recent years to design asynchronous algorithmic schemes for decentralized optimization.
Asynchronous gossip-based method under the edge-based random activation setting has been proposed by \cite{Boyd:2006} to solve averaging consensus problems. Later \cite{SL12c} extended this framework for solving \eqref{eqn:orgprob} and established almost surely convergence to the optimal solution when $f_i$ is smooth and convex. Most recently, \cite{xu2018convergence} also achieved almost surely convergence by iteratively activating a subset of agents.  
Besides (sub)gradient based methods, another well-known approach relies on solving the saddle point formulation of \eqref{eqn:orgprob} (see Section~\ref{sec:formulation} for the reformulation), where at each iteration a pair of primal and dual variables is updated alternatively.
The distributed ADMM (e.g., \cite{Bianchi-asyn,Wei-admm,zhang2015bi,bianchi2016coordinate}) has been studied in different asynchronous setting. 
More specifically, \cite{Bianchi-asyn,bianchi2016coordinate} randomly selected and updated a subset of agents iteratively where \cite{Bianchi-asyn} assuming $f_i$ being simple convex function and \cite{bianchi2016coordinate} establishing almost surely convergence for smooth convex objectives.
\cite{Wei-admm} employed the node-based random activation and achieved the ${\cal O}(1/\epsilon)$ rate of convergence
when $f_i$ is a simple convex function, and \cite{zhang2015bi} later established the same rate of convergence by activating one agent per iteration.     
Most recently, \cite{wu2016decentralized} proposed an asynchronous parallel primal-dual type method and established almost surely convergence when $f_i$ is smooth and convex.

Asynchronous decentralized algorithms discussed above require the knowledge of exact (sub)gradients (or function values) of $f_i$, however, this requirement is not realistic when dealing with minimization of generalized risk and online (streaming) data distributed over a network.
There exists limited research on asynchronous decentralized stochastic optimization (e.g., \cite{Nedic11,KS2011,Chang-Stochastic}), for which only noisy gradient information of functions $f_i$, $i = 1,\ldots,m$, can be easily computed. While asynchronous decentralized stochastic first-order methods \cite{Nedic11,KS2011} established error bounds when $f_i$ is (strongly) convex, \cite{Chang-Stochastic} achieved ${\cal O}(1/\epsilon^2)$ rate of convergence for smooth and convex problems. 


Recently \cite{lan2017communication} proposed a class of primal-dual type communication-efficient methods for decentralized stochastic optimization, which obtained the best-known $\mathcal{O}(1/\epsilon)$ (resp., $\mathcal{O}(1/\sqrt{\epsilon})$) communication complexity and the optimal $\mathcal{O}(1/\epsilon^2)$ (resp., $\mathcal{O}(1/\epsilon)$) sampling complexity for solving nonsmooth convex (resp., strongly convex) problems under the synchronous setting.
This class of communication-efficient methods requires two rounds of communication involving all network agents per iteration, and hence may incur huge synchronous delays. 
Moreover, it was proposed to solve decentralized nonsmooth problems so that its convergence property is not clear when applying it to solve decentralized problems satisfying \eqref{eqn:nonsmooth}.
Inspired by \cite{lan2017communication}, we aim to propose an asynchronous decentralized algorithmic framework to solve \eqref{eqn:orgprob} under a more general setting \eqref{eqn:nonsmooth} but still maintains the complexity bounds achieved in \cite{lan2017communication}. 
Our main contributions in this paper can be summarized as follows. 
Firstly, 
we introduce a doubly randomized primal-dual method, namely, asynchronous decentralized primal-dual (\algone) method, which randomly activates two agents per iteration, and hence  
two rounds of communication between the activated agent and its neighboring agents are performed.
This proposed method can find a stochastic $\epsilon$-optimal solution in terms of both the primal optimality gap and feasibility residual in $\mathcal{O}(1/\epsilon)$ communication rounds when the objective functions are simple convex such that the local proximal subproblems can be solved exactly. 

Secondly, we present a new asynchronous stochastic decentralized primal-dual type method, called asynchronous accelerated stochastic decentralized communication sliding (\algtwo) method,
for solving decentralized stochastic optimization problems.
 It should be pointed out that \algtwo~ is a unified algorithm that can be applied to solve a wild range of problems under the general setting of \eqref{eqn:nonsmooth}.  
In particular, only $\mathcal{O}(1/\epsilon)$ (resp., $\mathcal{O}(1/\sqrt{\epsilon})$)
communication rounds are required while agents perform a total of $\mathcal{O}(1/\epsilon^2)$ (resp., $\mathcal{O}(1/\epsilon)$) stochastic
(sub)gradient evaluations for general convex (resp., strongly convex) functions.
Moreover, the latter bounds, a.k.a. sampling complexities, of \algtwo~ can achieve a better dependence on the Lipschitz constant $L$ when the objective function contains a smooth component, i.e., $L>0$ in \eqref{eqn:nonsmooth}, than other existing decentralized stochastic first-order methods.
Only requiring the access to stochastic (sub)gradients at each iteration, \algtwo~ is particularly efficient for solving problems with $f_i:=\bbe_{\xi_i}[F_i(x;\xi_i)]$,
which provides a communication-efficient way to deal with streaming data and decentralized machine learning.
We summarized the achieved communication and sampling complexities in this paper in Table~\ref{tab_summary}.

\begin{table}[h]
\caption{Complexity bounds for obtaining a stochastic $\epsilon$-solution under asynchronous setting}
\label{tab_summary}
\centering
  \vspace{0.1in}
  \setlength\tabcolsep{1.3mm}
  \begin{tabular}{|c|cc|cc|}
  \hline
  {\bf Problem type: $f_i$} & \multicolumn{2}{|c|}{\bf Communication Complexity}&\multicolumn{2}{|c|}{\bf Sampling Complexity}\\
  \hline
   & Our results & Existing results& Our results & Existing results\\
   \hline
   Simple convex & \makecell{${\cal O}\{1/\epsilon\}$\\ {\tiny ADPD}}& \makecell{${\cal O}\{1/\epsilon\}$\\ \tiny Distributed-ADMM\cite{Wei-admm}}& NA & NA\\
   \hline
   Stochastic, convex & \makecell{${\cal O}\{1/\epsilon\}$ \\ \tiny AA-SDCS}& \makecell{${\cal O}\{1/\epsilon^2\}$ \\ \tiny proximal-gradient\cite{Chang-Stochastic}}&\makecell{${\cal O}\{1/\epsilon^2\}$ \\ \tiny AA-SDCS}& \makecell{${\cal O}\{1/\epsilon^2\}$ \\ \tiny proximal-gradient\cite{Chang-Stochastic}}\\
   \hline
   \makecell{Stochastic,\\ strongly convex} & \makecell{${\cal O}\{1/\sqrt \epsilon\}$ \\ \tiny AA-SDCS}& \makecell{${\cal O}\{1/\sqrt \epsilon\}$\\ \tiny synchronous-SDCS\cite{lan2017communication}} & \makecell{${\cal O}\{1/\epsilon\}$ \\ \tiny AA-SDCS}& \makecell{${\cal O}\{1/\epsilon\}$ \\ \tiny synchronous-SDCS\cite{lan2017communication}} \\
   \hline
   \makecell{Stochastic, convex,\\smooth + nonsmooth}\footnotemark[1]& \makecell{${\cal O}\{1/\epsilon\}$ \\ \tiny AA-SDCS}& \makecell{${\cal O}\{1/\epsilon^2\}$ \\ \tiny proximal-gradient\cite{Chang-Stochastic}\footnotemark[2]}& \makecell{${\cal O}\{\tfrac{M^2+\sigma^2}{\epsilon^2} + \tfrac{\sqrt{L}}{\epsilon}\}$ \\ \tiny AA-SDCS}& \makecell{${\cal O}\{\tfrac{\sigma^2}{\epsilon^2} + \tfrac{L}{\epsilon}\}$ \\ \tiny proximal-gradient\cite{Chang-Stochastic}\footnotemark[2]}\\
   \hline
   \makecell{Stochastic,\\ strong convex,\\smooth + nonsmooth}\footnotemark[3]& \makecell{${\cal O}\{1/\sqrt\epsilon\}$ \\ \tiny AA-SDCS}& NA & \makecell{${\cal O}\{\tfrac{M^2+\sigma^2}{\mu\epsilon} + \sqrt{\tfrac{L}{\mu\sqrt\epsilon}}\}$ \\ \tiny AA-SDCS}& NA\\
   \hline
   \end{tabular}
\end{table}
\footnotetext[1]{Here we refer to object functions satisfying the condition that $L, M>0$ in \eqref{eqn:nonsmooth}.}
\footnotetext[2]{The proximal-gradient method proposed in \cite{Chang-Stochastic} can only deal with the case that $f_i$ is a composite function such that it is the summation of smooth functions and a simple nonsmooth function (cf. a regularizer). }
\footnotetext[3]{Here we refer to object functions satisfying the condition that $\mu, L, M>0$ in \eqref{eqn:nonsmooth}.}
Thirdly, we demonstrate the advantages of the proposed methods through preliminary numerical experiments for solving decentralized support vector machine (SVM) problems with real data sets. 
For all testing problems, \algtwo~can significantly save CPU running time over existing state-of-the-art decentralized methods.

To the best of our knowledge, this is the first time that these asynchronous communication sliding algorithms, and the aforementioned separate complexity bounds on communication rounds and stochastic (sub)gradient evaluations under the asynchronous setting are presented in the literature.

This paper is organized as follows.
In Section~\ref{sec:formulation}, we introduce the problem formulation and provide some preliminaries on distance generating functions and prox-functions. 
We present our main asynchronous decentralized primal-dual framework and establish their convergence properties in Section~\ref{sec:alg}. 
Section \ref{sec:exp} is devoted to providing some preliminary numerical results to demonstrate the advantages of our proposed algorithms.
The proofs of the main theorems in Section \ref{sec:alg} are provided in Appendix~\ref{app:conv}. 

\noindent {\bf Notation and Terminologies.}
We denote by $\0b$ and $\1b$ the vector of all zeros and ones whose dimensions vary from the context.
The cardinality of a set $S$ is denoted by $|S|$.
We use $I_d$ to denote the identity matrix in $\mathbb{R}^{d\times d}$.
We use $A \otimes B$ for matrices $A \in \mathbb{R}^{n_1\times n_2}$ and $B \in \mathbb{R}^{m_1\times m_2}$ to denote their Kronecker product of size $\mathbb{R}^{n_1m_1\times n_2m_2}$.
For a matrix $A \in \mathbb{R}^{n\times m}$, we use $A_{i,j}$ to denote the entry of $i$-th row and $j$-th column.
For any $m \ge 1$, the set of integers $\{1,\ldots,m\}$ is denoted by $[m]$.

\section{Problem setup}\label{sec:formulation}
Consider a multi-agent network system whose communication is governed by an undirected graph $\Gc = (\Vc,\Ec)$,
where $\Vc = [m]$ indexes the set of agents,
and $\Ec \subseteq \Vc \times \Vc$ represents the pairs of communicating agents.
If there exists an edge
from agent $i$ to $j$ denote by $(i,j)$, agent $i$ may exchange information with agent $j$. Therefore, each agent $i \in \Vc$ can directly
receive (resp., send) information only from (resp., to) the agents in its neighborhood
$
N_i = \{j \in \Vc \mid (i,j)\in \Ec\} \cup \{i\},
$
where we assume that there always exists a self-loop $(i,i)$ for
all agents $i \in \Vc$, with no communication delay. 
The associated Laplacian ${\cal L} \in \mathbb{R}^{m\times m}$ of $\Gc$ is defined as
\begin{align}\label{def_Laplacian}
{\cal L}_{i,j} = \left\{
\begin{array}{ll}
|N_i|-1& \textrm{ if } i = j\\
-1 & \textrm{ if } i\neq j \textrm{ and } (i,j) \in \Ec\\
0 & \textrm{ otherwise.}
\end{array}
\right.
\end{align}

We introduce an individual copy $x_i$ of the decision variable $x$
for each agent $i \in \Vc$.
Hence, by employing the Laplacian matrix $\cal L$, \eqref{eqn:orgprob} can be written compactly as
\begin{align}\label{eqn:prob}
\min_{\xb \in X^m} &~F(\xb) := \tsum_{i=1}^m f_i(x_i)\\
\text{s.t. } &~\Lb\xb = \0b, \nn
\end{align}
where $X^m := X_1 \times \ldots \times X_m$, $\xb = (x_1; \ldots; x_m) \in X^m$, $F:X^m\to \mathbb{R}$, and $\Lb = {\cal L}\otimes I_d \in \mathbb{R}^{md\times md}$.
The constraint $\Lb\xb = \0b$ is a compact way of writing $x_i = x_j$ for all pairs $(i,j) \in \Ec$. In view of Theorem 4.2.12 in \cite{hom1991topics}, $\Lb$ is symmetric positive semidefinite and its null space coincides with the ``agreement'' subspace, i.e.,
$\Lb\1b = \1b^{\top}\Lb=\0b$.
To ensure each agents can obtain information from every other agents, we need the following assumption as a blanket assumption throughout the paper.
\begin{assumption}\label{assume:G}
The graph $\Gc$ is connected.
\end{assumption}
Under Assumption \ref{assume:G}, problem \eqref{eqn:orgprob} and \eqref{eqn:prob} are equivalent.
We next consider a reformulation of \eqref{eqn:prob}.
By the method of Lagrange multipliers,
problem \eqref{eqn:prob} is equivalent to the following saddle point problem:
\begin{align}\label{eqn:saddle}
\min_{\xb \in X^m}   \left[F(\xb) + \max_{\yb\in \mathbb{R}^{md}} \langle \Lb\xb, \yb\rangle\right],
\end{align}
where 
$\yb = (y_1; \ldots; y_m) \in \mathbb{R}^{md}$ are the Lagrange multipliers
associated with the constraints $\Lb\xb = \0b$.
We assume that there exists an optimal solution $\xb^* \in X^m$
of \eqref{eqn:prob}
and that there exists $\yb^* \in \mathbb{R}^{md}$ such that $(\xb^*,\yb^*)$
is a saddle point of \eqref{eqn:saddle}.
Finally, we define the following terminology. 
\begin{definition}\label{def_solution}
A point $\hat \xb \in X^m$ is called a stochastic $\epsilon$-solution of \eqref{eqn:prob} if
\begin{align}
\bbe[F(\hat \xb)-F(\xb^*)]\le \epsilon \text{ and } \bbe[\|\Lb\hat \xb\|]\le \epsilon.
\end{align}
We say that $\hat \xb$ has primal residual $\epsilon$ and feasibility residual $\epsilon$.
\end{definition}
Note that for problem \eqref{eqn:prob}, the feasibility residual $\|\Lb\hat\xb\|$ measures the disagreement among the local copies $\hat x_i$, for $i \in \Vc$. We will use these two criteria to evaluate the output solutions of the algorithms proposed in this paper.

\subsection{Distance generating function and prox-function}\label{subsec:bregman}
\textit{Prox-function}, also known as \textit{proximity control function} or \textit{Bregman distance function} \cite{BREGMAN1967}, has played an important role
as a substantial generalization of the Euclidean projection, since
it can be flexibly tailored to the geometry of a constraint set $U$.

For any convex set $U$ equipped with an arbitrary norm $\|\cdot\|_U$,
we say that a function $\o: U \to \mathbb{R}$ is a \textit{distance generating function} with modulus $1$
with respect to $\|\cdot\|_U$, if $\o$ is continuously differentiable and strongly convex with
modulus $1$ with respect to $\|\cdot\|_U$, i.e.,
$\la \nabla \o(x) - \nabla \o(u), x-u\ra \ge \|x-u\|_U^2, \ \forall x, u \in U.$
The prox-function induced by $\o$ is given by
\begin{align}\label{eqn:def_V}
V(x,u) \equiv V_{\o}(x,u) := \o(u) - [\o(x) + \la \nabla \o(x), u-x\ra].
\end{align}
We now assume that the constraint set $X_i$ 
for each agent in \eqref{eqn:orgprob} is equipped with norm $\|\cdot\|_{X_i}$, and
its associated prox-function is given by $V_i(\cdot,\cdot)$. It then follows from the strong convexity of $\o$ that
\begin{align}\label{eqn:def_V_i_s}
V_i(x_i,u_i) \ge \tfrac{1}{2}\|x_i-u_i\|_{X_i}^2, \quad \forall x_i, u_i \in X_i, \ i=1,\ldots,m.
\end{align}
We also define the norm associated with the primal feasible set $X^m=X_1 \times \ldots \times X_m$ of \eqref{eqn:saddle} as
$
\|\xb\|^2\equiv \|\xb\|_{X^m}^2:=\tsum_{i=1}^m \|x_i\|_{X_i}^2, \forall
\xb=(x_1;\ldots;x_m)\in X^m$. Therefore, the associated prox-function $\mathbf{V}(\cdot,\cdot)$ is defined as
$
\mathbf{V}(\xb,\ub):=\tsum_{i=1}^m V_i(x_i,u_i), \ \forall \xb,\ub\in X^m.
$
In view of \eqref{eqn:def_V_i_s} 
\beq\label{eqn:def_V_s}
\mathbf{V}(\xb,\ub)\ge \tfrac{1}{2}\|\xb-\ub\|^2, \ \forall \xb, \ub \in X^m.
\eeq

Throughout the paper, we endow the dual space where the multipliers $\yb$ of \eqref{eqn:saddle} reside
with the standard Euclidean norm $\|\cdot\|_2$, since the feasible region of $\yb$ is unbounded.
For simplicity, we often write $\|\yb\|$ instead of $\|\yb\|_2$ for a dual multiplier $\yb \in  \mathbb{R}^{md}$.

\section{The algorithms}\label{sec:alg}
In this section, we introduce an asynchronous decentralized primal-dual framework for solving \eqref{eqn:orgprob} in the decentralized setting. Specifically, two asynchronous methods are presented, namely asynchronous decentralized primal-dual method in Subsection~\ref{sec:adpd} and asynchronous accelerated stochastic decentralized communication sliding in Subsection~\ref{sec:asdcs}, respectively.
Moreover, we establish complexity bounds (number of inter-node communication rounds and/or intra-node stochastic (sub)gradient evaluations) separately in terms of primal functional optimality gap and constraint (or consistency) violation for solving \eqref{eqn:orgprob}-\eqref{eqn:prob}. 

\subsection{Asynchronous decentralized primal-dual method}\label{sec:adpd}
Our main goals in this subsection are to introduce the basic scheme of asynchronous decentralized primal-dual (\algone) method, as well as establishing its complexity results. Throughout this subsection, we assume that $f_i$ is a simple function such that we can solve the primal subproblem \eqref{eqn:algo6-i} explicitly.

We formally present the \algone~method in Algorithm~\ref{alg:ADPD-i}. Each agent $i$ maintains two local sequences, namely, the primal estimates $\{x_i^k\}$
and the dual variables $\{y_i^k\}$.
All primal estimates $x_i^{-1}$ and $x_i^0$ are locally initialized from some arbitrary point in $X_i$, and each dual variable $y_i^0=\0b$.
At each iteration $k \ge 1$, only one randomly selected agent (cf. activated agent) $i_k\in [m]$ updates its dual variable $y_{i_k}^k$, and then one randomly selected agent $j_k\in [m]$ updates its primal variable $x_{j_k}^k$.
In particular, each agent in the activated agent's neighborhood, i.e., agents $i \in N_{i_k}$,
computes a local prediction $\tilde{x}_i^k$ using the two previous primal estimates (ref. \eqref{eqn:algo1-i}), and send it to agent $i_k$.
In \eqref{eqn:algo2-i}-\eqref{eqn:algo3-i},
the activated agent $i_k$ calculates its neighborhood disagreement $v_{i_k}^k$ using the receiving messages,
and updates the dual variable $y_{i_k}^k$. Other agents' dual variables remain unchanged. 
Then, another round of communication \eqref{eqn:algo5-i} between the activated agent $j_k$ and its neighboring agents occurs after the dual prediction step \eqref{eqn:algo4-i}. 
Lastly, the activated agent $j_k$ solves the proximal projection subproblem \eqref{eqn:algo6-i} to update $x_{j_k}^k$, and other agents' primal estimates remain the same as the last iteration.

It should be emphasized that each iteration $k$ only involves two communication rounds (cf. \eqref{eqn:algo2-i} and \eqref{eqn:algo5-i}) between the activated agents and its neighboring agents, which significantly reduces synchronous delays appearing in many decentralized methods (e.g., \cite{Duchi2012,Wilbur-ADMM,WYin-Extra,lan2017communication}), since these methods require at least one communication round between all agents and their neighboring agents iteratively. 
Also note that similar to the asynchronous ADMM proposed in \cite{Wei-admm}, \algone~employs node-based activation. However, while \cite{Wei-admm} requires all agents to update dual variables iteratively based on the information obtaining from communication, in \algone~only the activated agent $i_k$ needs to collect neighboring information and update its dual variable (see \eqref{eqn:algo2-i} and \eqref{eqn:algo3-i}), and hence \algone~further reduces communication costs and synchronous delays comparing to \cite{Wei-admm}. Moreover, \algone~can achieve the same rate of convergence ${\cal O}(1/\epsilon)$ as \cite{Wei-admm} under the assumption that \eqref{eqn:algo6-i} can be solved explicitly.
We will demonstrate later that by exploiting the strong convexity, an improved ${\cal O}(1/\sqrt \epsilon)$ rate of convergence can be obtained.

\begin{algorithm}
\caption{Asynchronous decentralized primal-dual (\algone) update for each agent $i$}
\label{alg:ADPD-i}
\begin{algorithmic}
\STATE Let $x_i^0 = x_i^{-1}\in X_i$ and $y_i^0 = \0b$ for $i \in [m]$, the nonnegative parameters $\{\alpha_k\}$, $\{\tau_k\}$ and $\{\eta_k\}$
be given.
\FOR{$k = 1, \ldots, N$}
\STATE{Uniformly choose $i_k,j_k\in [m]$, and update $(x_i^k,y_i^k)$ according to}
\begin{align}
\tilde{x}_i^k =&~ \alpha_k (x_i^{k-1}-x_i^{k-2}) + x_i^{k-1}.\label{eqn:algo1-i}\\
v_{i_k}^k = &~\tsum_{j\in N_{i_k}}{\cal L}_{i_k,j}\tilde{x}_j^k. \COMMENT{\text{communication}}\label{eqn:algo2-i}\\
y_i^k =&~
\begin{cases}
\argmin_{y_i \in \mathbb{R}^d}~  \la -v_i^k, y_i\ra + \tfrac{\tau_k}{2}\|y_i-y_i^{k-1}\|^2=y_i^{k-1} +\tfrac{1}{\tau_k}v_i^k, \ &i=i_k,\\
y_i^{k-1}, \ &i \ne i_k.
\end{cases}\label{eqn:algo3-i}\\
\tilde{y}_i^k =&~ m(y_i^k-y_i^{k-1})+y_i^{k-1}.\label{eqn:algo4-i}\\
w_{j_k}^k = &~\tsum_{j\in N_{j_k}}{\cal L}_{j_k,j}\tilde y_j^k.\COMMENT{\text{communication}}\label{eqn:algo5-i}\\
x_i^k = &~
\begin{cases} 
\argmin_{x_i\in X_i}~  \la w_i^k, x_i\ra + f_i(x_i) + \eta_kV_i(x_i^{k-1},x_i), \ &i=j_k,\\
x_j^{k-1}, \ &i\ne j_k.
\end{cases}\label{eqn:algo6-i}
\end{align}
\ENDFOR
\end{algorithmic}
\end{algorithm}

In the following theorem, we provide a specific selection of $\{\alpha_k\}$, $\{\tau_k\}$ and $\{\eta_k\}$, which leads to ${\cal O}(1/\epsilon)$ complexity bounds for the functional optimality gap and also the feasibility residual to obtain a stochastic $\epsilon$-solution of \eqref{eqn:prob}.

\vgap
\begin{theorem}\label{main_adpd}
Let $\xb^*$ be an optimal solution of \eqref{eqn:prob}, and $d_{max}$ be the maximum degree of graph $\Gc$, and suppose that $\{\alpha_k\}$, $\{\tau_k\}$ and $\{\eta_k\}$ are set to 
\beq\label{para_adpd}
\alpha_k = m,\ \eta_k = 2md_{max}, \mbox{ and } \tau_k=2md_{max}, \forall k= 1,\dots,N.
\eeq
Then, for any $N\ge 1$, we have
\begin{align}\label{bnd_adpd}
\bbe_{[i_k,j_k]}\{F(\bar\xb^N)-F(\xb^*)\}&
\le {\cal O}\left\{\tfrac{m\Delta_{\xb^0}}{N+m} \right\}, \ \
\bbe_{[i_k,j_k]}\{\|\Lb\bar\xb^N\|\}\le {\cal O}\left\{\tfrac{m\Delta_{\xb^0}}{N+m} \right\},
\end{align}
where $\bar\xb^N=\tfrac{1}{N+m}(\tsum_{k=0}^{N-1}\xb^k+m\xb^N)$, $\{\xb^k\}$ is generated by Algorithm~\ref{alg:ADPD-i}, and $\Delta_{\xb^0}:=\max\Big\{C_{\xb^0},\|\Lb\xb^0\|+ md_{max}\Big(\|\yb^*\|+ \sqrt{\tfrac{C_{\xb^0} + \la \Lb\xb^0,\yb^*\ra}{md_{max}}}
 \Big)\Big\}$ with $C_{\xb^0} = F(\xb^0)-F(\xb^*)+md_{max}\mathbf{V}(\xb^0,\xb^*)$.
\end{theorem}

Theorem~\ref{main_adpd} implies the total number of inter-node communication rounds performed by \algone~to find a stochastic $\epsilon$-solution of \eqref{eqn:prob} can be bounded by 
\beq\label{comp_communication}
{\cal O}\left\{\tfrac{md_{max}\Delta_{\xb^0}}{\epsilon}\right\}.
\eeq

Observed that in Algorithm~\ref{alg:ADPD-i}, we assume that $f_i$'s are simple functions such that \eqref{eqn:algo6-i} can be solved explicitly. However, since $f_i$'s are possibly nonsmooth functions and/or possess composite structures, it is often difficult to solve \eqref{eqn:algo6-i} especially when $f_i$ is provided in the form of expectation. In the next subsection, we present a new asynchronous stochastic decentralized primal-dual type method, called the asynchronous accelerated stochastic decentralized communication sliding (\algtwo) method, for the case when \eqref{eqn:algo6-i} is not easy to solve. 

\subsection{Asynchronous accelerated stochastic decentralized communication sliding} \label{sec:asdcs}
In the subsection, we show that one can still maintain the same number of inter-node communications
even when the subproblem \eqref{eqn:algo6-i} is approximately solved through an optimal
stochastic approximation method, namely AC-SA proposed in \cite{ghadimi2013optimal,GhaLan10-1,Lan10-3}, and that
the total number of required stochastic (sub)gradient evaluations (or sampling complexity) is comparable to centralized mirror descent methods. 
Throughout this subsection, we assume that only noisy (sub)gradient information of $f_i$,  $i = 1, \ldots, m$, is available or easier to compute.
This situation happens when the function $f_i$'s are given either in the form of expectation
or as the summation of lots of components.
Moreover, we assume that the first-order information of the function $f_i$, $i = 1, \ldots, m$,
can be accessed by a stochastic oracle (SO), which,  given a point $u^t \in X$,
outputs a vector $G_i(u^t,\xi_i^t)$ such that
\begin{align}
&\bbe[G_i(u^t,\xi_i^t)]=f_i'(u^t) \in \partial f_i(u^t),\label{assume:unbiased}\\
&\bbe[\|G_i(u^t,\xi_i^t)-f_i'(u^t)\|_*^2]\le \sigma^2, \label{assume:sm_bounded}
\end{align}
where $\xi_i^t$ is a random vector which models a source of uncertainty and is independent of the search point $u^t$, and the distribution $\mathbb{P}(\xi_i)$ is not known in advance.
We call $G_i(u^t,\xi_i^t)$ a \textit{stochastic (sub)gradient} of $f_i$ at $u^t$. Observe that this assumption covers the case that one can access the exact (sub)gradients of $f_i$ whenever $\sigma =0$.

In order to exploit the strong convexity of the prox-function $V_i$, we assume in this subsection that each prox-function $V_i(\cdot,\cdot)$ (cf. \eqref{eqn:def_V}) are growing quadratically 
with the \textit{quadratic growth constant} $\mathcal{C}$, i.e., there exists a constant $\mathcal{C}>0$ such that
\begin{align}\label{eqn:proxquad}
V_i(x_i,u_i) \le \tfrac{\mathcal{C}}{2}\|x_i-u_i\|_{X_i }^2, \quad \forall x_i, u_i \in X_i,\ i=1,\ldots,m.
\end{align}
By \eqref{eqn:def_V_i_s}, we must have $\mathcal{C}\ge 1$.


\begin{algorithm}
\caption{Asynchronous Accelerated Stochastic Decentralized Communication Sliding (\algtwo)}
\label{alg:AASDCS}
\begin{algorithmic}
\STATE{Let $x_i^0 = x_i^{-1} = \underline{x}_i^0 \in X_i$, $y_i^0 =\0b$ for $i \in [m]$ and the nonnegative parameters $\{\alpha_k\}$, $\{\tau_k\}$, $\{\eta_k\}$ and $\{T_k\}$ be given.}
\FOR{$k = 1, \ldots, N$}
\STATE{Uniformly choose $i_k,j_k\in [m]$, and update $(\underline x_i^k, y_i^k)$ according to}
\begin{align}
\tilde{x}_i^k =&~ \alpha_k [m\underline{x}_i^{k-1}-(m-1)\underline x_i^{k-2}-x_i^{k-2}] +  x_i^{k-1}.\label{eqn:algo1-acs}\\
v_{i_k}^k = &~\tsum_{j\in N_{i_k}}{\cal L}_{i_k,j}\tilde{x}_j^k.\COMMENT{\text{communication}}\label{eqn:algo2-acs}\\
y_i^k =&~
\begin{cases}
y_i^{k-1} +\tfrac{1}{\tau_k}v_i^k, \ &i=i_k,\\
y_i^{k-1}, \ &i \ne i_k.
\end{cases}\label{eqn:algo3-acs}\\
\tilde{y}_i^k =&~ m(y_i^k-y_i^{k-1})+y_i^{k-1}.\label{eqn:algo4-acs}\\
w_{j_k}^k = &~\tsum_{j\in N_{j_k}}{\cal L}_{j_k,j}\tilde y_j^k.\COMMENT{\text{communication}}\label{eqn:algo5-acs}\\
(x_i^k,\underline{x}_i^k) =&~ 
\begin{cases}
\inner(f_i,X_i,V_i,T_k,\eta_k,w_i^k,x_i^{k-1}), \ &i = j_k,\\
(x_i^{k-1}, \underline x_i^{k-1}), \ &i\ne j_k.
\end{cases}\label{eqn:algo6-acs}
\end{align}
\ENDFOR

\STATE
\STATE The \inner~ (Accelerated Communication-Sliding) procedure called at \eqref{eqn:algo6-acs} is stated as follows.\\
\textbf{procedure:} $(x,\underline{x}) =\text{\inner}(\phi,U,V,T,\eta,w,x)$
\STATE Let $u^0 = \underline{u}^0 = x$ and the parameters $\{\beta_t\}$ and $\{\lambda_t\}$ be given.
\FOR{$t = 1,\ldots,T$}
\STATE
\begin{align}
\hat u^t =&~ \tfrac{(1-\lambda_t)(\mu+\eta+\beta_t)}{\beta_t+(1-\lambda_t^2)(\mu+\eta)}\underline u^{t-1} + \tfrac{\lambda_t[(1-\lambda_t)(\mu +\eta)+ \beta_t]}{\beta_t+(1-\lambda_t^2)(\mu+\eta)} u^{t-1}.\label{eqn:inner0}\\
G^t =&~ G(\hat u^t,\xi^t). \COMMENT{\text{Call the SO}}\label{eqn:gradient}\\ 
u^t = &~ \argmin_{u \in U}\left\{\lambda_t[\langle w + G^t + \eta(\nabla w(\hat u^t)- \nabla w(x)), u \rangle +(\mu + \eta) V(\hat u^t,u)]\right.\nn\\
&\quad \quad\quad  \left.+ [(1-\lambda_t)(\mu+\eta) + \beta_t]V(u^{t-1},u)\right\}.\label{eqn:inner1}\\
\underline u^t =&~ (1-\lambda_t)\underline u^{t-1} +  \lambda_t u^t.\label{eqn:inner2}
\end{align}
\ENDFOR
\STATE Set $x = u^T$ and $\underline{x} = \underline{u}^T$.\\
\textbf{end procedure}
\end{algorithmic}
\end{algorithm}

We now add a few comments about Algorithm~\ref{alg:AASDCS}. 
Firstly, similar to SDCS proposed in \cite{lan2017communication}, \algtwo~exploits two loops: the doubly randomized primal-dual scheme as outer loop and the \inner~procedure as inner loop. More specifically, \algtwo~utilizes the AC-SA method proposed in \cite{ghadimi2013optimal,GhaLan10-1,Lan10-3} to approximately solve the primal subproblem in \eqref{eqn:algo6-i}, which provides a unified scheme for solving a general class of problems defined in \eqref{eqn:nonsmooth} and leads to accelerated rate of convergence when $f_i$ possesses smooth structure. 
Secondly, the same dual information $w = w_{j_k}^k$ (see \eqref{eqn:algo5-acs})
has been used throughout the $T =T_k$ iterations of the \inner~procedure,
and hence no additional communication is required within the procedure. 
Finally, since \algtwo~randomly selects one subproblem \eqref{eqn:algo6-i} and solved it inexactly, the outer loop also needs to be carefully designed to attain the best possible rate of convergence. In fact, the \inner~procedure provides two approximate solutions of \eqref{eqn:algo6-i}: one is the primal estimate $\{x_i^k\}$ and the other is $\{\underline x_i^k\}$, which will be maintained by each agent and later play a crucial role in the development and convergence analysis of \algtwo. We also accordingly modify the primal extrapolation step of the outer loop (cf. \eqref{eqn:algo1-acs}).
For later convenience, we refer to the subproblem \inner~solved at iteration $k$ as $\Phi^k(x_i)$, i.e., 
\begin{align}\label{def_phi}
\argmin_{x_i\in X_i}\left\{\Phi^k(x_i):= \la w_i^k, x_i\ra + f_i(x_i) + \eta_kV_i(x_i^{k-1},x_i)\right\}.
\end{align}

Theorem~\ref{main_aasdcs} provides a specific selection of $\{\alpha_k\}$, $\{\tau_k\}$, $\{\eta_k\}$ and $\{T_k\}$ for Algorithm~\ref{alg:AASDCS}, and $\{\lambda_t\}$ and $\{\beta_t\}$ for the \inner~procedure, which leads to ${\cal O}(1/\epsilon)$ complexity bounds for the functional optimality gap and also the feasibility residual to obtain a stochastic $\epsilon$-solution of \eqref{eqn:prob}.

\begin{theorem}\label{main_aasdcs}
Let $\xb^*$ be an optimal solution of \eqref{eqn:prob}, and $d_{max}$ be the maximum degree of graph $\Gc$, and suppose that the parameters $\{\lambda_t\}$ and $\{\beta_t\}$ in the \inner~procedure of Algorithm~\ref{alg:AASDCS} be set to 
\beq\label{para_sgd1}
\lambda_t = \tfrac{2}{t+1}, \ \beta_t = \tfrac{4(\mathcal{C}+L)}{t(t+1)}, \ \forall t\ge1,
\eeq
and 
$\{\alpha_k\}$, $\{\tau_k\}$, $\{\eta_k\}$ and $\{T_k\}$ are set to 
\begin{align}\label{para_aasdcs}
\alpha_k &= 1,\ \eta_k = 4md_{max},\ \tau_k=2d_{max},\nn\\
\mbox{and} \ T_k&=\max\left\{\left\lceil \tfrac{(M^2+\sigma^2)N}{d_{max}\mathcal{D}}\right\rceil, \left\lceil \sqrt{\tfrac{\mathcal C+L}{md_{max}}}\right\rceil\right\}, \ \forall k= 1,\dots,N,
\end{align}
for some ${\cal D}>0$. Then, for any $N\ge 1$, we have
\begin{align}\label{bnd_aasdcs}
\bbe\{F(\bar\xb^N)-F(\xb^*)\}&\le {\cal O}\left\{\tfrac{m\Delta_{\xb^0,{\cal D}}}{N+m}\right\}, \ \
\bbe\{\|\Lb\bar\xb^N\|\}\le {\cal O}\left\{\tfrac{m\Delta_{\xb^0,{\cal D}}}{N+m}\right\},
\end{align}
where $\bar\xb^N=\tfrac{1}{N+m}(\tsum_{k=0}^{N-1}\underline \xb^k+m\underline \xb^N)$, $\{\underline \xb^k\}$ is generated by Algorithm~\ref{alg:AASDCS}, and $\Delta_{\xb^0,{\cal D}}: = \max\Big\{C_{\xb^0,{\cal D}}, \|\Lb\xb^0\|+ d_{max}\Big(\|\yb^*\|+
\sqrt{\tfrac{C_{\xb^0,{\cal D}} + \la \Lb\xb^0,\yb^*\ra}{d_{max}}}
\Big)\Big\}$ with $C_{\xb^0,{\cal D}}=F(\xb^0)-F(\xb^*)+md_{max}\mathbf{V}(\xb^0,\xb^*) + \tfrac{\cal D}{m}$.
\end{theorem}

In view of Theorem~\ref{main_aasdcs}, letting ${\cal D}={\cal O} (m^2d_{max})$, we can see that the total number of inter-node communication rounds and intra-node (sub)gradient evaluations required by \algtwo~for finding a stochastic $\epsilon$-solution of \eqref{eqn:prob} can be bounded by 
\beq\label{comp_aasdcs}
{\cal O}\left\{\tfrac{md_{max}\Delta_{\xb^0,{\cal D}}}{\epsilon}\right\} \ \mbox{and} \ 
{\cal O}\left\{\tfrac{(M^2+\sigma^2)\Delta_{\xb^0,{\cal D}}^2}{\epsilon^2d_{max}^2}+ \sqrt{\tfrac{m({\cal C} +L)}{d_{max}}}\tfrac{\Delta_{\xb^0,{\cal D}}}{\epsilon}\right\},
\eeq
respectively.
It also needs to be emphasized that the sampling complexity (second bound in \eqref{comp_aasdcs}) only sublinearly depends on the Lipschitz constant $L$. 

Now consider the case when $f_i$'s are strongly convex (i.e.,  $\mu > 0$ in \eqref{eqn:nonsmooth}).
The following theorem instantiates Algorithm~\ref{alg:AASDCS} by providing a selection of $\{\alpha_k\}$, $\{\tau_k\}$, $\{\eta_k\}$ and $\{T_k\}$, which leads to a improved ${\cal O}(1/\sqrt \epsilon)$ complexity bound for the functional optimality gap and also the feasibility residual to obtain a stochastic $\epsilon$-solution of \eqref{eqn:prob}.

\begin{theorem}\label{main_aasdcs_s}
Let $\xb^*$ be an optimal solution of \eqref{eqn:prob}, and $d_{max}$ be the maximum degree of graph $\Gc$, and suppose that the parameters $\{\lambda_t\}$ and $\{\beta_t\}$ in the \inner~procedure of Algorithm~\ref{alg:AASDCS} be set to \eqref{para_sgd1},
and 
$\{\alpha_k\}$, $\{\tau_k\}$, $\{\eta_k\}$ and $\{T_k\}$ are set to 
\begin{align}\label{para_aasdcs_s}
\alpha_k &= \tfrac{k+3m-1}{k+3m},\ \eta_k = \tfrac{(k+3m-1)\mu}{2}- \tfrac{\mathcal{C}+L}{T_k(T_k+1)},\ \tau_k=\tfrac{32md^2_{max}}{(k+3m)\mu},\nn\\
\mbox{and} \ T_k&=\max\left\{\left\lceil \tfrac{64m(M^2+\sigma^2)N}{\mathcal{D}\mu^2}\right\rceil, \left\lceil \sqrt{\tfrac{4(\mathcal C+L)}{(k+3m-3)\mu}}\right\rceil\right\}, \ \forall k= 1,\dots,N.
\end{align}
Then, for any $N\ge 1$, we have
\begin{align}\label{bnd_aasdcs_s}
\bbe\{F(\bar\xb^N)-F(\xb^*)\}&\le {\cal O}\left\{\tfrac{m^2\Delta_{\xb^0,{\cal D},\mu}}{m^2+N^2}\right\}, \ \ 
\bbe\{\|\Lb\bar\xb^N\|\}\le {\cal O}\left\{\tfrac{m^2\Delta_{\xb^0,{\cal D},\mu}}{m^2+N^2}\right\},
\end{align}
where $\bar\xb^N=\tfrac{2}{6m^2+N(N+6m+1)}(\tsum_{k=0}^{N-1}(k+2m+1)\underline \xb^k+ m(N+3m)\underline \xb^N)$, $\{\underline \xb^k\}$ is generated by Algorithm~\ref{alg:AASDCS}, and $\Delta_{\xb^0,{\cal D},\mu}: = \max\Big\{C_{\xb^0,{\cal D},\mu}, \|\Lb\xb^0\|+ \tfrac{d^2_{max}\|\yb^*\|}{\mu}+ d_{max}\sqrt{\tfrac{C_{\xb^0,{\cal D},\mu}+\la \Lb\xb^0,\yb^*\ra}{\mu}}
\Big\}$ with $C_{\xb^0,{\cal D},\mu} = F(\xb^0)-F(\xb^*)+m\mu\mathbf{V}(\xb^0,\xb^*) + \tfrac{{\cal D}\mu}{m^2}$.
\end{theorem}
As a consequence of Theorem~\ref{main_aasdcs_s}, letting ${\cal D }= {\cal O}(m^3)$, we can see that the total number of inter-node communication rounds and intra-node (sub)gradient evaluations required by \algtwo~for finding a stochastic $\epsilon$-solution of \eqref{eqn:prob}, respectively, can be bounded by 
\beq\label{comp_aasdcs_s}
{\cal O}\left\{md_{max}\sqrt{\tfrac{\Delta_{\xb^0,{\cal D},\mu}}{\epsilon}}\right\}, \ \mbox{and} \ 
{\cal O}\left\{\tfrac{(M^2+\sigma^2)\Delta_{\xb^0,{\cal D},\mu}}{\mu^2\epsilon} + \sqrt{\tfrac{m({\cal C}+L)}{\mu}}\left(\tfrac{\Delta_{\xb^0,{\cal D},\mu}}{\epsilon}\right)^{1/4}\right\}.
\eeq

\section{Numerical experiments}\label{sec:exp}
We demonstrate the advantages of our proposed \algtwo~ method over the state-of-art synchronous algorithm, stochastic decentralized communication sliding (SDCS) method, proposed in \cite{lan2017communication} through some preliminary numerical experiments. 

Let us consider the decentralized linear Support Vector Machines (SVM) model with the following hinge loss function
\beq\label{svm}
\max\{0,1-v\la x,u\ra\},
\eeq
where $(v,u)\in \bbr\times\bbr^d$ is the pair of class label and feature vector, and $x\in \bbr^d$ denotes the weight vector. 
We consider two types of stochastic decentralized linear SVM problems in this paper. 
For the convex case, we study $1$-norm SVM problem \cite{zhu20041,bradley1998feature} defined in \eqref{sto_svm},
while for the strongly convex case, we study $2$-norm SVM model defined in \eqref{dsc_svm}. 
Moreover, we use the Erhos-Renyi algorithm
\footnote{We implemented the Erhos-Renyi algorithm based on a MATLAB function written by Pablo Blider, which can be found in 
\url{https://www.mathworks.com/matlabcentral/fileexchange/4206}.}
to generate the underlying decentralized network.   Note that nodes with different degrees are drawn in different colors 
(cf. Figure\ref{network8}). 
We also used the real dataset named ``ijcnn1'' from LIBSVM\footnote{This real dataset can be downloaded from \url{https://www.csie.ntu.edu.tw/~cjlin/libsvmtools/datasets/}.} and drew $40,000$ samples from this dataset as our problem instance data to train the decentralized linear SVM model. These samples are evenly split over the network agents. For example, if we have $m=8$ nodes (or agents) in the decentralized network (see Figure~\ref{network8}), each network agent has $5,000$ samples. 

\begin{figure}[H]
\centering
\includegraphics[scale = 0.21]{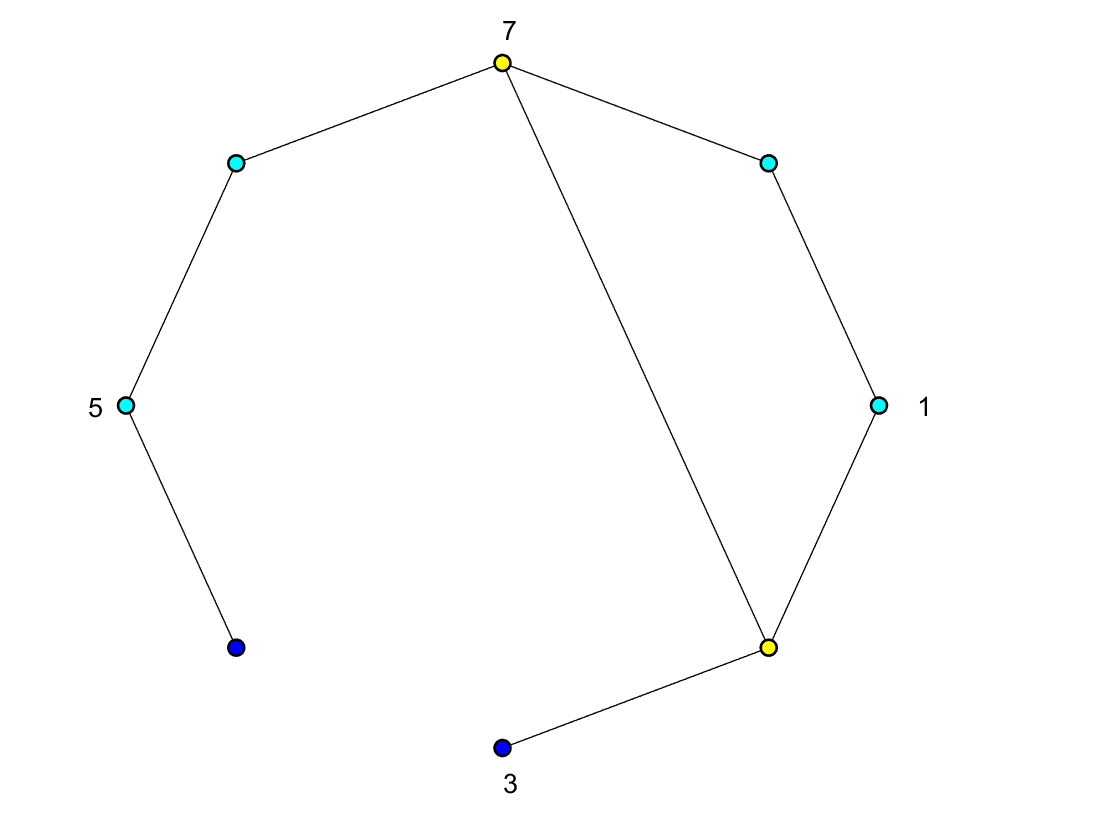}
\caption{\scriptsize The $8$-agent decentralized network randomly generated by Erhos-Renyi algorithm.}
\label{network8}
\end{figure}

With the same initial points $x^0=\0b$ and $y^0 = \0b$, we compare the performances of our algorithms with the SDCS method \cite{lan2017communication} for solving \eqref{eqn:orgprob}-\eqref{eqn:prob} by reporting the progresses of objective function values and feasibility residuals $\|\Lb\xb\|$ versus the elapsed CPU running time (in seconds) for solving the aforementioned two different types of problems. In all problem instances, we use $\|\cdot\|_2$ norm in both the primal and dual spaces, and hence in the parameter settings of SDCS $\|\Lb\|$ refers to the maximum eigenvalue of the Laplacian matrix $\cal L$. Moreover, all algorithms are implemented in MATLAB R2016a and run in the computer environment of with $32$-core (Intel(R) Xeon(R) CPU E5-2673 v3 \@ $2.40$GHz) virtual machine on Microsoft Azure. 
Since the underlying network has $8$ agents, we utilized the parallel toolbox in MATLAB to simulate the synchronous setting for SDCS. However, inter-node communication is instant and no delay is simulated in all experiments. In fact, such simulation setup is in favor of the synchronous methods, since these methods can be heavily slowed down by different processing speeds of the agents (cores) and inter-node communication speeds. 

\vgap
\noindent{\bf Convex case: decentralized $1$-norm SVM \cite{zhu20041,bradley1998feature}}
Consider a stochastic decentralized linear SVM problem defined over the $m$-agent decentralized network as
\begin{align}\label{sto_svm}
\min_{\xb} \tsum_{i=1}^m &\left[f_i(x_i):= \bbe_{(v_i,u_i)}[\max\{0,1-v_i\la x_i,u_i\ra\}] + \tfrac{1}{\|\mathcal{S}_i\|}\|x_i\|_1\right]\\
\text{s.t. } &~\Lb\xb = \0b,\nn
\end{align}
where 
$(v_i,u_i)$ represents a uniform random variable with support $\mathcal{S}_i$ and $\mathcal{S}_i$ denotes the dataset belonging to node $i$. We compare the performances of \algtwo~with SDCS for the decentralized network setups, $m=8$ (cf. R. Figure~\ref{network8}). For all problem instances, we choose the parameters of \algtwo~as in Theorem~\ref{main_aasdcs}. For SDCS, we choose parameters as suggested in \cite{lan2017communication}.
  \begin{figure}[H]
  \centering
  \includegraphics[scale = 0.2]{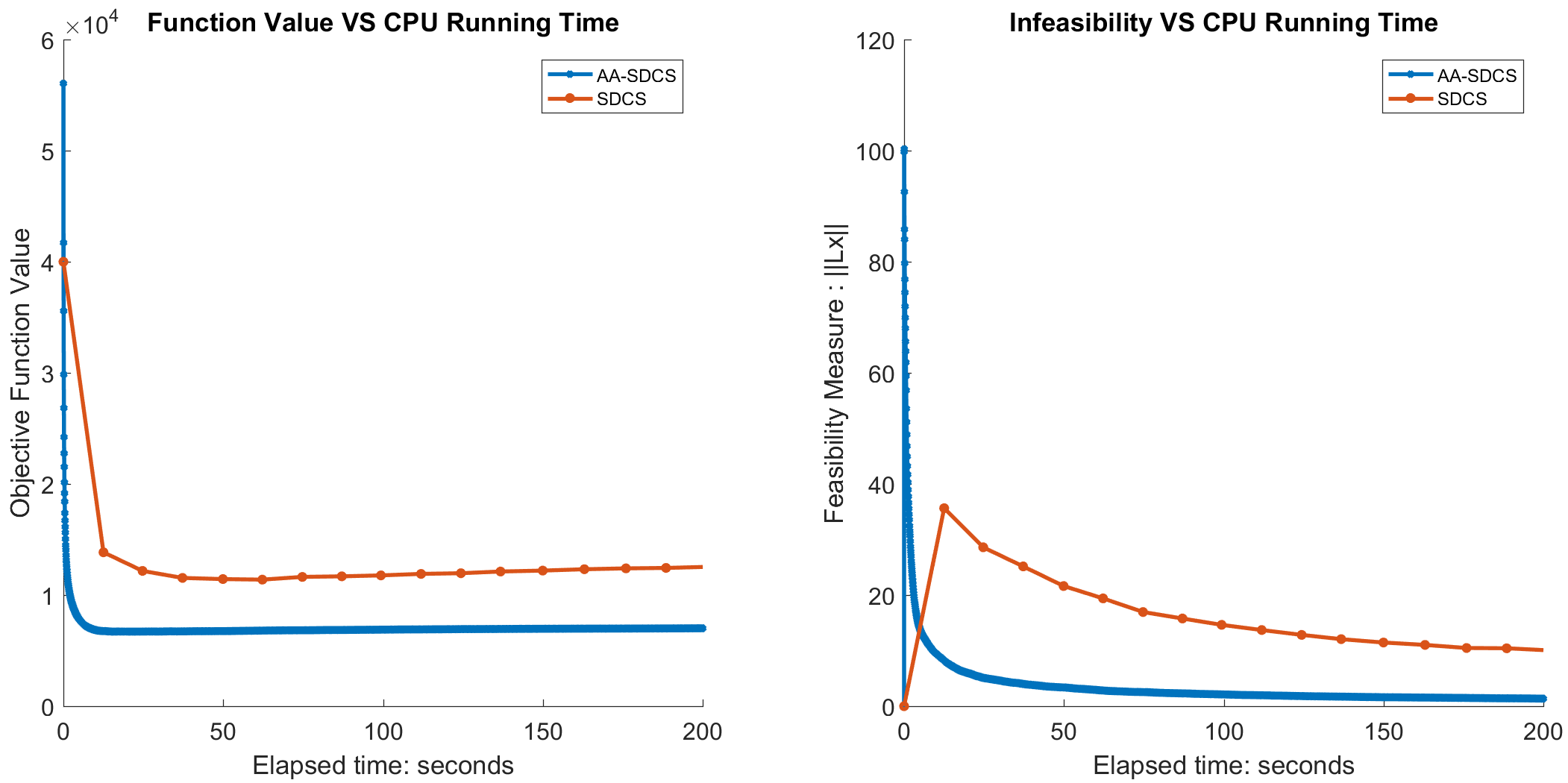}
  \caption{\scriptsize $1$-norm SVM defined over the $8$-agent decentralized network (cf. Figure~\ref{network8}), and we report the progresses of the objective function values on the left and the feasibility residuals on the right versus the elapsed CPU running time in seconds.}
  \label{l1g8}
  \end{figure}

In Figure~\ref{l1g8}, the vertical-axis of the left subgraph represents the objective function values, the vertical-axis of the right subgraph represents the feasibility measure $\|\Lb\xb\|$, and the horizontal-axis is the elapsed CPU running time in seconds. These numerical results are consistent with our theoretical analysis. We also need to emphasize that \algtwo~can significantly save CPU running time over SDCS in terms of both objective function values and feasibility residuals as shown in Figure~\ref{l1g8} even when each agent (Core) has the same processing speed. 

\vgap
\noindent{\bf Strongly convex case: decentralized $2$-norm SVM}
Consider a decentralized linear SVM problem with $l_2$ regularizer defined over the $m$-agent decentralized network as the following
\begin{align}\label{dsc_svm}
\min_{\xb} ~\tsum_{i=1}^m &\left[f_i(x_i):= \bbe_{(v_i,u_i)}[\max\{0,1-v_i\la x_i,u_i\ra\}] + \tfrac{1}{2|\mathcal{S}_i|}\|x_i\|^2_2\right]\\
\text{s.t. } &~\Lb\xb = \0b.\nonumber
\end{align}
We compare the performances of \algtwo~ with SDCS for the decentralized network setups, $m=8$ (cf. R. Figure~\ref{network8}). For all problem instances, 
we choose the parameters of \algtwo~ as in Theorem~\ref{main_aasdcs_s}. For SDCS, we choose parameters as suggested in \cite{lan2017communication}.
\begin{figure}[H]
  \centering
  \includegraphics[scale = 0.2]{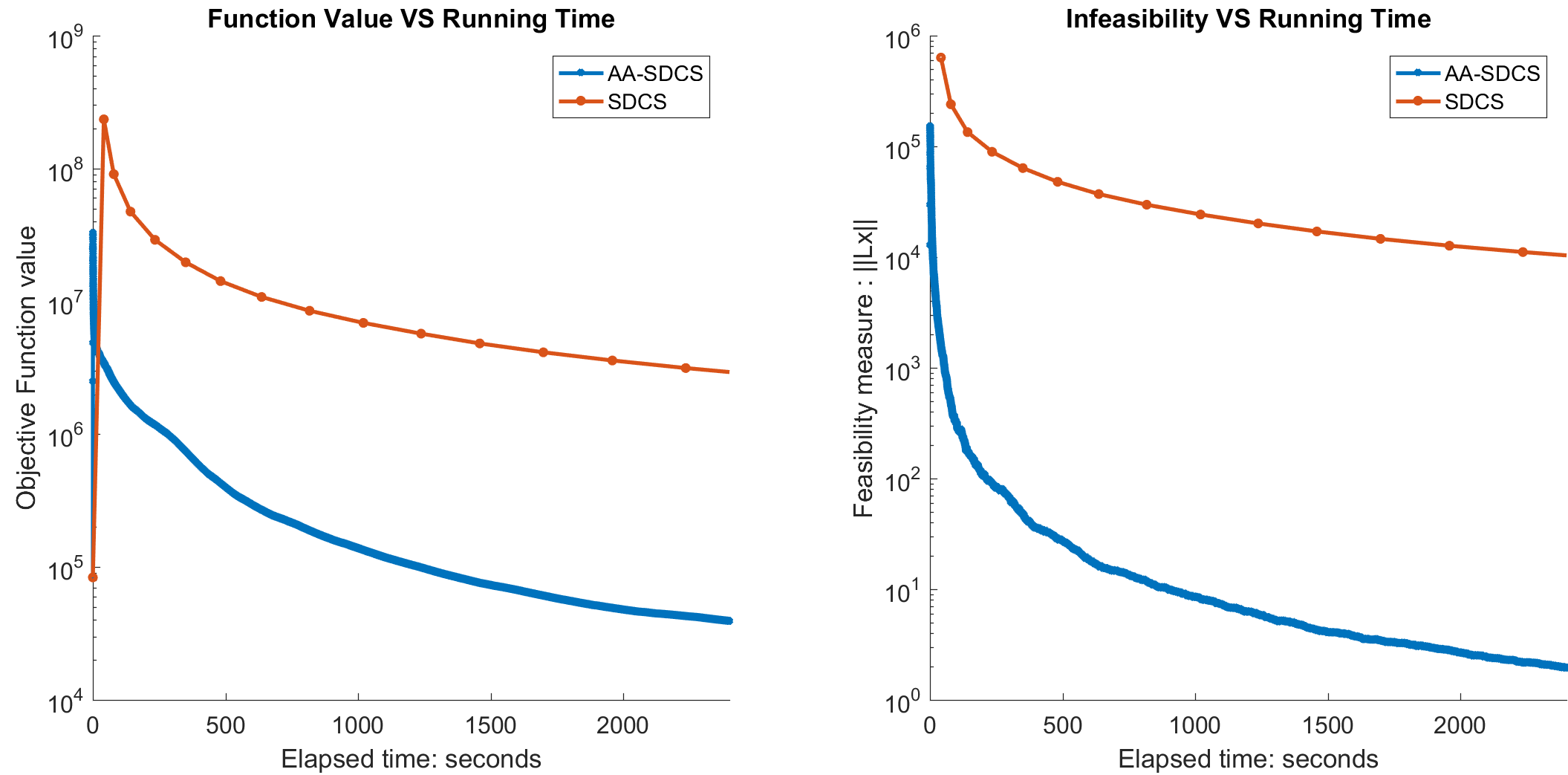}
  \caption{\scriptsize $2$-norm SVM defined over the $8$-agent decentralized network (cf. Figure~\ref{network8}), and we report the progresses of the objective function values on the left and the feasibility residuals on the right versus the elapsed CPU running time in seconds.}
  \label{l2g8}
  \end{figure}
  The above figures clearly show that \algtwo~can significantly save CPU running time over SDCS in terms of both objective function values and feasibility residuals. Moreover, comparing Figure~\ref{l2g8} with Figure~\ref{l1g8}, we can find out \algtwo~ obtains more improvements over SDCS for solving decentralized $2$-norm SVM problems than decentralized $1$-norm SVM problems. In fact, the decentralized $2$-norm SVM problem defined in \eqref{dsc_svm} has a composite objective structure that consists of a nonsmooth hinge loss function and a smooth strongly convex $l2$-regularizer, and the convergence results of \algtwo~ has a better dependence on the Lipschitz constant $L$, which
  indicates that it can obtain a faster convergence speed than SDCS for solving decentralized $2$-norm SVM problems.

\bibliographystyle{plain}
\bibliography{NIPS2018}
\newpage
\appendix
\section{Convergence analysis}\label{app:conv}
In this section, we provide detailed convergence analysis of \algone~(cf. Algorithm~\ref{alg:ADPD-i}) and \algtwo~(cf. Algorithm~\ref{alg:AASDCS}) presented in Section~\ref{sec:alg}.
\subsection{Some basic tools: gap functions, termination criteria and technical results}\label{sec:tools}
Given a pair of feasible solutions $\zb = (\xb,\yb)$  and $\bar{\zb} = (\bar{\xb},\bar{\yb})$ of \eqref{eqn:saddle},
we define the \textit{primal-dual gap function} $Q(\zb;\bar{\zb})$ by
\begin{align}\label{eqn:gap}
Q(\zb;\bar{\zb}) := &~F(\xb) + \la \Lb\xb, \bar{\yb} \ra - [F(\bar{\xb}) + \la \Lb\bar{\xb}, \yb \ra].
\end{align}
Sometimes we also use the notations $Q(\zb;\bar{\zb}) := Q(\xb, \yb;\bar{\xb},\bar{\yb})$
or $Q(\zb;\bar{\zb}) := Q(\xb,\yb;\bar{\zb}) = Q(\zb;\bar{\xb},\bar{\yb})$.
One can easily see that $Q(\zb^*;\zb) \le 0$ and $Q(\zb;\zb^*) \ge 0$ for all $\zb \in X^m \times \mathbb{R}^{md}$, where $\zb^* = (\xb^*,\yb^*)$ is a saddle point of \eqref{eqn:saddle}.
For compact sets $X^m \subset \mathbb{R}^{md}$, $Y \subset \mathbb{R}^{md}$, the gap function
\begin{align} \label{eqn:gapQ0}
\sup_{\bar{\zb}\in X^m \times Y} Q(\zb;\bar{\zb})
\end{align}
measures the accuracy of the approximate solution $\zb$ to the saddle point problem \eqref{eqn:saddle}.

However, the saddle point formulation \eqref{eqn:saddle} of our problem of interest \eqref{eqn:orgprob}
may have an unbounded feasible set.
We adopt the perturbation-based termination criterion by Monteiro and Svaiter \cite{Monteiro01,Monteiro02,Monteiro03}
and propose a modified version of the gap function in \eqref{eqn:gapQ0}.
More specifically, we define
\begin{align}\label{eqn:gapp}
g_Y(\vb,\zb) := \sup_{\bar{\yb}\in Y} Q(\zb;\xb^*,\bar{\yb}) - \la \vb, \bar{\yb}\ra,
\end{align}
for any closed set $Y \subseteq \mathbb{R}^{md}$, $\zb \in X^m\times \mathbb{R}^{md}$ and $\vb \in \bbr^{md}$.
If $Y = \mathbb{R}^{md}$, we omit the subscript $Y$ and simply use the notation $g(\vb,\zb)$.
This perturbed gap function allows us to bound the objective function value and the feasibility separately.

In the following proposition, we adopt a result from \cite[Proposition 2.1]{GL-AADMM} to describe the relationship between the perturbed gap function \eqref{eqn:gapp} and the approximate solutions (see Definition~\ref{def_solution}) to problem \eqref{eqn:prob}.

\begin{proposition}\label{prop:approx}
For any $Y \subset \mathbb{R}^{md}$ such that $\0b \in Y$,
if $g_Y(\Lb\xb,\zb) \le \epsilon < \infty$ and $\|\Lb\xb\| \le \delta$, where $\zb = (\xb,\yb) \in X^m \times \mathbb{R}^{md}$, then $\xb$ is an $(\epsilon,\delta)$-solution of \eqref{eqn:prob}.
In particular, when $Y = \mathbb{R}^{md}$, for any $\vb$ such that $g(\vb,\zb) \le \epsilon < \infty$
and $\|\vb\|\le \delta$, we always have $\vb = \Lb\xb$.
\end{proposition}

Although the proposition was originally developed for deterministic cases, the extension of this to stochastic cases is straightforward. In fact, if we define $g_Y(\vb,\zb)$ as follows
\begin{align*}
g_Y(\vb,\zb) := \sup_{\bar{\yb}\in Y} \bbe[Q(\zb;\xb^*,\bar{\yb}) - \la \vb, \bar{\yb}\ra],
\end{align*}
from the definition of $Q$, we have $g_Y(\vb,\zb)=\sup_{\bar \yb\in Y}\bbe[F(\xb)-F(\xb^*)-\langle \Lb\xb-\vb,\bar \yb\rangle]$ due to $\Lb\xb^*=0$. Therefore, when $Y=\bbr^{md}$, the results in Proposition~\ref{prop:approx} holds, since $g_Y(\vb,\zb)$ is bounded for any $\bar \yb\in Y$.

We also define some auxiliary notations which play important roles in the convergence analysis. Let $\hat{\xb}^k$, $\hat{\yb}^k$, $\hat \xb^k_{+}$ and $\hat{\underline \xb}^k$ be defined as follows, $\forall t=1,\dots,k$
\begin{align}
\hat{\xb}^k&=\argmin_{\xb\in X^m}\langle \Lb\tilde{\yb}^k, \xb\rangle+F(\xb)+\eta_t\mathbf{V}(\xb^{k-1},\xb),\label{def_xhat}\\
\hat{\yb}^k&=\yb^{k-1} + \tfrac{1}{\tau_k}\Lb\tilde \xb^k,\label{def_yhat}\\ 
(\hat \xb^k_{+},\hat{\underline \xb}^k) &= \inner(F,X^m,\mathbf{V},T_k,\eta_k,\Lb\tilde \yb^k,\xb^{k-1}),\label{def_xhatu} 
\end{align}
and $\hat{\xb}^0= \hat \xb^k_{+} = \hat{\underline \xb}^0=\xb^0, \hat{\yb}^0=\yb^0=\0b$. Note that some notations may be abused in the above definitions, since $\xb^k, \tilde \yb^k, \yb^k, \tilde \xb^k$ can be generated by both Algorithm~\ref{alg:ADPD-i} and Algorithm~\ref{alg:AASDCS}. However, these definitions become clear when we refer to them in the convergence analysis of certain algorithm. For example, when we refer to $\hat \xb^k$ in the convergence analysis of Algorithm~\ref{alg:ADPD-i}, notations $\tilde \yb^k$ and $\xb^{k-1}$ in its definition clearly refer to \eqref{eqn:algo4-i} and \eqref{eqn:algo6-i} in Algorithm~\ref{alg:ADPD-i}. 

The following lemma below characterizes the solution of the primal and dual projection steps
\eqref{eqn:algo6-i}, \eqref{eqn:algo3-i}, \eqref{eqn:algo3-acs} (also \eqref{def_xhat}, \eqref{def_yhat}) as well as the projection in inner loop \eqref{eqn:inner2}. The proof of this result can be found in Lemma 2 of \cite{GhaLan10-1}.

\begin{lemma}\label{lem:optimality}
  Let the convex function $q: U \to \mathbb{R}$, the points $\bar{x},\bar{y}\in U$
  and the scalars $\mu_1,\mu_2 \in \mathbb{R}$ be given. Let $\omega: U\rightarrow \bbr$ be a differentiable convex function and $V(x,z)$ be defined in \eqref{eqn:def_V}. If
  \begin{align*}
  u^* \in \argmin\left\{ q(u) + \mu_1V(\bar{x}, u) + \mu_2V(\bar{y},u): u\in U\right\},
  \end{align*}
  then for any $u \in U$, we have
  \begin{align*}
  q(u^*) +\mu_1V(\bar{x},u^*) + \mu_2V(\bar{y},u^*)
  &\le q(u) +\mu_1V(\bar{x},u)+ \mu_2V(\bar{y},u) - (\mu_1 + \mu_2)V(u^*,u).
  \end{align*}
\end{lemma}

For any given weight sequence $\{\hat \theta_k\}$ such that $\hat \theta_k\ge 0,\ \ \sum_{k=0}^{N}\hat \theta_k=1$, let $\{\theta_k\}$ be defined as
\begin{align}\label{def_xweight}
\theta_k=
\begin{cases}
\hat \theta_0-(m-1)\hat \theta_1,\ \  & k=0,\\
m\hat \theta_k-(m-1)\hat \theta_{k+1}, \ \ & k=1, \dots, N-1,\\
m\hat \theta_N\ \ & k=N.
\end{cases}
\end{align}
Therefore, $\tsum_{k=0}^N\theta_k=\tsum_{k=0}^N\hat \theta_k = 1$.
In the following lemma, we provide some important relations that will be used later in the convergence analysis.

\begin{lemma}\label{lem_exp}
For weight sequence $\{\theta_k\}$ defined as in \eqref{def_xweight} and any $\xb\in X^m,\yb\in \bbr^{md}$, we have
\begin{align*}
\bbe_{[i_k,j_k]}\left\{\tsum_{k=0}^N\theta_k[F(\xb^k)-F(\xb)+\langle \Lb\xb^k,\yb\rangle]\right\} &= \bbe_{[i_k,j_k]}\left\{\tsum_{k=1}^N\hat \theta_k[F(\hat\xb^k)-F(\xb)+\langle \Lb\hat\xb^k,\yb\rangle]\right\},\nn\\
\bbe_{[i_k,j_k]}\left\{\tsum_{k=0}^N\theta_k[F(\underline \xb^k)-F(\xb)+\langle \Lb\underline \xb^k,\yb\rangle]\right\} &= \bbe_{[i_k,j_k]}\left\{\tsum_{k=0}^N\hat \theta_k[F(\hat{\underline \xb}^k)-F(\xb)+\langle \Lb\hat {\underline\xb}^k,\yb\rangle]\right\},\\
\bbe_{[i_k,j_k]}\{\mathbf{V}(\hat \xb^k,\xb)\}&=\bbe_{[i_k,j_k]}\{m\mathbf{V}(\xb^k,\xb)-(m-1)\mathbf{V}(\xb^{k-1},\xb)\},\\
\bbe_{[i_k,j_k]}\{\mathbf{V}(\hat\xb^{k},\xb^{k-1})\}&=\bbe_{[i_k,j_k]}\{mV_{j_k}(x_{j_k}^k,x_{j_k}^{k-1})\},
\end{align*}
where $\bbe_{[i_k,j_k]}$ represents taking expectation over $i_1,j_1,\dots,i_k,j_k$, and $\xb^k$, $\hat \xb^k$, $\underline \xb^k$ and $\hat{\underline \xb}^k$ are defined in \eqref{eqn:algo6-i}
, \eqref{def_xhat}, \eqref{eqn:algo6-acs}, \eqref{def_xhatu} respectively. 
\end{lemma}
\begin{proof}
Note that by \eqref{def_xhat} and the fact that $j_k$ is chosen uniformly from $\{1,\dots,m\}$, we have
\begin{align}\label{exp_jk}
  &\bbe_{j_k}\{F(\xb^k)-F(\xb)+\langle \Lb\xb^k,\yb\rangle\}\nn\\
   &~= (1-\tfrac{1}{m})[F(\xb^{k-1})-F(\xb)+\langle \Lb\xb^{k-1},\yb\rangle]+\tfrac{1}{m}[F(\hat\xb^{k-1})-F(\xb)+\langle \Lb\hat\xb^{k-1},\yb\rangle].
\end{align}
Therefore, by \eqref{def_xweight} we obtain
\begin{align*}
\bbe_{[i_k,j_k]}&\left\{\tsum_{k=0}^N\theta_k[F(\xb^k)-F(\xb)+\langle \Lb\xb^k,\yb\rangle]\right\}\\
&=\bbe_{[i_k]}\bbe_{[j_k]}\{(\hat\theta_0-(m-1)\hat\theta_1)[F(\xb^0)-F(\xb)+\langle \Lb\xb^0,\yb\rangle]|[i_k]\}\\
& \quad +\bbe_{[i_k]}\bbe_{[j_k]}\left\{\tsum_{k=1}^{N-1}(m\hat\theta_k-(m-1)\hat\theta_{k+1})[F(\xb^k)-F(\xb)+\langle \Lb\xb^k,\yb\rangle|[i_k]]\right\}\\
& \quad + \bbe_{[i_k]}\bbe_{[j_k]}\left\{m\hat\theta_N[F(\xb^N)-F(\xb)+\langle \Lb\xb^N,\yb\rangle]|[i_k]\right\}\\
& = \bbe_{[i_k,j_k]}\left\{\hat\theta_0[F(\hat\xb^0)-F(\xb)+\langle \Lb\hat\xb^0,\yb\rangle]\right\}\\
&\quad + \bbe_{[i_k]}\bbe_{[j_k]}\left\{\tsum_{k=1}^Nm\hat\theta_k[F(\xb^{k})-F(\xb)+\langle \Lb\xb^{k},\yb\rangle]|[i_k]\right\}\\
&\quad - \bbe_{[i_k]}\bbe_{[j_k]}\left\{\tsum_{k=1}^N(m-1)\hat\theta_k[F(\xb^{k-1})-F(\xb)+\langle \Lb\xb^{k-1},\yb\rangle]|[i_k]\right\}\\
&= \bbe_{[i_k,j_k]}\left\{\tsum_{k=0}^N\hat\theta_k[F(\hat\xb^k)-F(\xb)+\langle \Lb\hat\xb^k,\yb\rangle]\right\},
\end{align*}
where the last equality is obtained by applying \eqref{exp_jk} and rearranging the terms. Similarly, in view of \eqref{def_xhatu}, 
we have
\begin{align*}
  &\bbe_{j_k}\{F(\underline \xb^k)-F(\xb)+\langle \Lb\underline \xb^k,\yb\rangle\}\nn\\
   &~= (1-\tfrac{1}{m})[F(\underline \xb^{k-1})-F(\xb)+\langle \Lb\underline \xb^{k-1},\yb\rangle]+\tfrac{1}{m}[F(\hat{\underline \xb}^{k-1})-F(\xb)+\langle \Lb\hat{\underline\xb}^{k-1},\yb\rangle],
\end{align*}
and hence 
the second identity follows from the same argument. Moreover, for any $\xb\in X^m$, $k\ge 1$, we have
\begin{align*}
\bbe_{[i_k,j_k]}\{\mathbf{V}(\xb^k,\xb)\}&=\bbe_{[i_k,j_k]}\left\{\tsum_{j=1}^mV_j(x_j^k,x_j)\right\}\\
&=\bbe_{[i_k,j_k]}\left\{\tfrac{1}{m}\mathbf{V}(\hat\xb^{k},\xb) +(1-\tfrac{1}{m})\mathbf{V}(\xb^{k-1},\xb)\right\},
\end{align*}
where the first equality follows from the definition of $\mathbf{V}(\cdot,\cdot)$, and the last equality follows by taking expectation on $j_k$. Similarly, we can obtain the last relation of the lemma.
\end{proof}
We define $\hat \yb^k$ (see \eqref{def_yhat}) and $\hat \xb_{+}^k$ (see \eqref{def_xhatu}) in a similar way as $\hat \xb^k$, and hence, following the same technique as in the above lemma, we can conclude
\begin{align*}
\bbe_{[i_k,j_k]}\{\tilde\yb^k\}&=\bbe_{[i_k,j_k]}\{\hat\yb^k\},\\
\bbe_{[i_k,j_k]}\{\|\yb-\hat\yb^k\|^2\}&=\bbe_{[i_k,j_k]}\{m\|\yb-\yb^k\|^2-(m-1)\|\yb-\yb^{k-1}\|^2\},\\
\bbe_{[i_k,j_k]}\{\|\yb^{k-1}-\hat\yb^{k}\|^2\}&=\bbe_{[i_k,j_k]}\{m\|y_{i_k}^{k-1}-y_{i_k}^{k}\|^2\},
\end{align*}
and 
\begin{align*}
\bbe_{[i_k,j_k]}\{\mathbf{V}(\hat \xb^k_{+},\xb)\}&=\bbe_{[i_k,j_k]}\{m\mathbf{V}(\xb^k,\xb)-(m-1)\mathbf{V}(\xb^{k-1},\xb)\},
\end{align*}
where $\xb^{k}$ (cf. \eqref{eqn:algo6-acs}) is generated by Algorithm~\ref{alg:AASDCS}.

\subsection{Convergence properties of Algorithm~\ref{alg:ADPD-i} (A-DPD)}
We now provide an important recursion relation of Algorithm~\ref{alg:ADPD-i} in the following lemma.
\begin{lemma}\label{lem_recursion}
Let the gap function $Q$ be defined as in \eqref{eqn:gap}, and $\bar \zb^N:=(\bar \xb^N,\bar \yb^N)=\tsum_{k=0}^N(\theta_k\xb^k,\hat\theta_k\hat\yb^k)$, where $\{\theta_k\}$ is a nonnegative sequence that satisfies \eqref{def_xweight}. Also let $\xb^k$ and $\yb^k$ be defined in \eqref{eqn:algo6-i} and \eqref{eqn:algo3-i}, respectively. Then for any $k\ge 1$, we have
\begin{align}\label{Q_rec}
\bbe_{[i_k,j_k]}\{Q(\bar\zb^N;\zb)\}&\le \hat\theta_0Q_0(\xb,\yb)  + \bbe_{[i_k,j_k]}\left\{\tsum_{k=1}^N\hat\theta_k\langle \Lb(m\xb^k-(m-1)\xb^{k-1}-\tilde\xb^k),\yb-\tilde\yb^{k}\rangle\right\} \nonumber\\
&\quad + \bbe_{[i_k,j_k]}\left\{\tsum_{k=1}^Nm\hat\theta_k\eta_k[\mathbf{V}(\xb^{k-1},\xb)-\mathbf{V}(\xb^k,\xb)-V_{j_k}(x_{j_k}^{k-1},x_{j_k}^k)]\right\}\nn\\
&\quad + \bbe_{[i_k,j_k]}\left\{\tsum_{k=1}^N\tfrac{m\hat\theta_k\tau_k}{2}[\|\yb-\yb^{k-1}\|^2-\|\yb-\yb^k\|^2-\|y_{i_k}^{k-1}-y_{i_k}^k\|^2]\right\},
\end{align}
where $Q_0(\xb,\yb)$ is defined as 
\beq\label{def_Q0}
Q_0(\xb,\yb):=F(\xb^0)-F(\xb)+\langle\Lb\xb^0,\yb\rangle.
\eeq
\end{lemma} 
\begin{proof}
By the definitions of $Q(\cdot;\cdot)$ in \eqref{eqn:gap} and $\bar\zb^N$, we have
\begin{align*}
Q(\bar \zb^N;\zb)&=F(\bar \xb^N)-F(\xb)+\langle \Lb\bar\xb^N,\yb\rangle-\langle\Lb\xb,\bar \yb^N\rangle\\
&\le \tsum_{k=0}^N\theta_k[F(\xb^k)-F(\xb)+\langle \Lb\xb^k,\yb\rangle]-\tsum_{k=0}^N\hat\theta_k\langle\Lb\xb,\hat \yb^k\rangle,
\end{align*}
where the inequality follows from the convexity of $F(\cdot)$. By taking expectation over $i_1,j_1,\dots,i_k,j_k$ and applying Lemma~\ref{lem_exp}, we obtain
\begin{align*}
\bbe_{[i_k,j_k]}\left\{Q(\bar \zb^N;\zb)\right\}&\le\bbe_{[i_k,j_K]}\left\{\tsum_{k=0}^N\hat \theta_k[F(\hat\xb^k)-F(\xb)+\langle\Lb\hat\xb^k,\yb\rangle - \langle \Lb\xb,\hat\yb^k\rangle]\right\}.
\end{align*}
Note that by applying Lemma~\ref{lem:optimality} to \eqref{def_xhat} and \eqref{def_yhat}, we have
\begin{align}\label{condi_subp}
  \langle \Lb\tilde\yb^k,\hat\xb^k-\xb\rangle + F(\hat\xb^k)- F(\xb)&\le \eta_k[\mathbf{V}(\xb^{k-1},\xb)-\mathbf{V}(\hat \xb^{k},\xb)-\mathbf{V}(\xb^{k-1},\hat \xb^{k})],\\
\langle \Lb\tilde\xb^k,\yb-\hat\yb^k\rangle&\le \tfrac{\tau_k}{2}[\|\yb-\yb^{k-1}\|^2-\|\yb-\hat\yb^{k}\|^2-\|\yb^{k-1}-\hat\yb^{k}\|^2].\nn
\end{align}
Combining the above three inequalities and in view of $\hat \yb^0=\yb^0=\0b$, we can conclude that
\begin{align}\label{raw_bnd_Q}
\bbe_{[i_k,j_k]}\{Q(\bar\zb^N;\zb)\}&\le \hat\theta_0Q_0(\xb,\yb) + \bbe_{[i_k,j_k]}\left\{\sum_{k=1}^N\hat\theta_k[\langle\Lb\tilde\yb^k,\xb-\hat\xb^k\rangle+\langle\Lb(\hat\xb^k-\tilde\xb^k),\yb\rangle+\langle\Lb(\tilde\xb^k-\xb),\hat\yb^k\rangle]\right\}\nn\\
&\quad +\bbe_{[i_k,j_k]}\left\{\tsum_{k=1}^N\hat\theta_k\eta_k[\mathbf{V}(\xb^{k-1},\xb)-\mathbf{V}(\hat\xb^k,\xb)-\mathbf{V}(\xb^{k-1},\hat\xb^{k})]\right\}\nn\\
&\quad +\bbe_{[i_k,j_k]}\left\{\tsum_{k=1}^N\tfrac{\hat\theta_k\tau_k}{2}[\|\yb-\yb^{k-1}\|^2-\|\yb-\hat\yb^{k}\|^2-\|\yb^{k-1}-\hat\yb^{k}\|^2]\right\}\nn\\
&\le \hat\theta_0Q_0 + \bbe_{[i_k,j_k]}\left\{\tsum_{k=1}^N\hat\theta_k\langle\Lb(\hat\xb^k-\tilde\xb^k),\yb-\tilde\yb^k\rangle\right\}\nn\\
&\quad + \bbe_{[i_k,j_k]}\left\{\tsum_{k=1}^Nm\hat\theta_k\eta_k[\mathbf{V}(\xb^{k-1},\xb)-\mathbf{V}(\xb^k,\xb)-V_{j_k}(x_{j_k}^{k-1},x_{j_k}^k)]\right\}\nn\\
&\quad + \bbe_{[i_k,j_k]}\left\{\tsum_{k=1}^N\tfrac{m\hat\theta_k\tau_k}{2}[\|\yb-\yb^{k-1}\|^2-\|\yb-\yb^k\|^2-\|y_{i_k}^{k-1}-y_{i_k}^k\|^2]\right\},
\end{align}
where the second inequality follows from Lemma~\ref{lem_exp} and the result in \eqref{Q_rec} immediately follows from taking expectation on $j_k$.
\end{proof}

The following proposition establishes the main convergence property of the A-DPD method stated in Algorithm~\ref{alg:ADPD-i}.
\begin{proposition}\label{prop_bndQ}
Let the iterates $(\xb^k, \hat\yb^k)$, $k=1,\dots,N$, be generated by Algorithm~\ref{alg:ADPD-i} and be defined as in \eqref{def_yhat}, respectively, and let $\bar \zb^N:=(\tsum_{k=0}^N\theta_k\xb^k,\tsum_{k=0}^N\hat\theta_k\hat\yb^k)$. Assume that the parameters $\{\alpha_k\}$, $\{\tau_k\}$, and $\{\eta_k\}$ in Algorithm~\ref{alg:ADPD-i} satisfy 
\begin{align}
 \hat\theta_{k}\tau_{k}= &~\hat\theta_{k-1}\tau_{k-1}, \quad k = 2, \ldots, N,\label{theta_tau}\\
\hat\theta_{k}\eta_{k} \le &~\hat\theta_{k-1}\eta_{k-1}, \quad k = 2, \ldots, N,\label{theta_eta}\\
  \alpha_{k}\hat\theta_{k}=&~m\hat\theta_{k-1}, \quad k = 2, \ldots, N+1,\label{alpha_theta}\\
  4m\alpha_kd_{max}^2 \le &~\eta_{k-1}\tau_k, \quad k = 2, \ldots, N,\label{eta_tau_alpha1}\\
   4(m-1)^2d_{max}^2 \le &~\eta_{k}\tau_k, \quad k = 1, \ldots, N.\label{eta_tau_alpha2}
\end{align}
where $\{\hat \theta_k\}$ is some given weight sequence and $d_{max}$ is the maximum degree of graph $\Gc$. Then, for any $\zb:=(\xb,\yb)\in X^m\times\bbr^{md}$, we have
\beq\label{bnd_Q1}
\bbe_{[i_k,j_k]}\{Q(\bar\zb^N;\zb)\}\le \hat\theta_0(F(\xb^0)-F(\xb)) + m\hat\theta_1\eta_1\mathbf{V}(\xb^0,\xb)+ \bbe_{[i_k,j_k]}\langle \vb,\yb\rangle,
\eeq
where $Q$ is defined in \eqref{eqn:gap}, and $\vb$ is defined as 
\beq\label{def_resi}
\vb:=\hat\theta_0 \Lb\xb^0+ m\hat\theta_N\Lb(\xb^N-\xb^{N-1})+m\hat\theta_1\tau_1\yb^N.
\eeq
Furthermore, for any saddle point $(\xb^*,\yb^*)$ of \eqref{eqn:saddle}, we have
\begin{align}\label{bnd_xyn}
\tfrac{m\hat\theta_N}{2}\left(1-\tfrac{4d_{max}^2}{\eta_N\tau_N}\right)&\max\left\{\tfrac{\eta_N}{2}\bbe_{[i_k,j_k]}\|x_{j_N}^N-x_{j_N}^{N-1}\|^2,\tau_N\bbe_{[i_k,j_k]}\|\yb^*-\yb^N\|^2\right\}\nn\\
&\le \hat\theta_0(F(\xb^0)-F(\xb^*) + \langle\Lb\xb^0,\yb^*\rangle) + m\hat\theta_1\eta_1\mathbf{V}(\xb^0,\xb^*) + \tfrac{m\hat\theta_1\tau_1}{2}\|\yb^*\|^2.
\end{align}
\end{proposition}
\begin{proof}
In view of Lemma~\ref{lem_recursion}, we have
\begin{align}\label{rel1}
\bbe_{[i_k,j_k]}\{Q(\bar\zb^N;\zb)\}&\le \hat\theta_0Q_0(\xb,\yb)  + \bbe_{[i_k,j_k]}\left\{\tsum_{k=1}^N\hat\theta_k\Delta_k\right\},
\end{align}
where 
\begin{align}\label{def_Delta}
\Delta_k:= &\langle \Lb(m\xb^k-(m-1)\xb^{k-1}-\tilde\xb^k),\yb-\tilde\yb^{k}\rangle 
+ m\eta_k[\mathbf{V}(\xb^{k-1},\xb)-\mathbf{V}(\xb^k,\xb)-V_{j_k}(x_{j_k}^{k-1},x_{j_k}^k)]\nn\\
&\quad + \tfrac{m\tau_k}{2}[\|\yb-\yb^{k-1}\|^2-\|\yb-\yb^k\|^2-\|y_{i_k}^{k-1}-y_{i_k}^k\|^2].
\end{align}
Now we will provide a bound for $\tsum_{k=1}^N\hat\theta_k\Delta_k$. Observe that by \eqref{eqn:algo1-i}, we obtain
\begin{align*}
\tsum_{k=1}^N\hat\theta_k&\langle \Lb(m\xb^k-(m-1)\xb^{k-1}-\tilde \xb^k),\yb-\tilde\yb^k\rangle\\
&= \tsum_{k=1}^N\hat\theta_k\langle \Lb(m(\xb^k-\xb^{k-1})-\alpha_k(\xb^{k-1}-\xb^{k-2})),\yb-\tilde \yb^k\rangle\\
&= \tsum_{k=1}^N\left[m\hat\theta_k\langle\Lb(\xb^k-\xb^{k-1}),\yb-\tilde\yb^k\rangle - \hat\theta_k\alpha_k\langle \Lb(\xb^{k-1}-\xb^{k-2}),\yb-\tilde\yb^{k-1}\rangle\right]\\
&\quad+ \tsum_{k=1}^N\hat\theta_k\alpha_k\langle \Lb(\xb^{k-1}-\xb^{k-2}),\tilde\yb^k-\tilde \yb^{k-1})\rangle\\
&=m\hat\theta_N\langle \Lb(\xb^N-\xb^{N-1}),\yb-\yb^N-(m-1)(\yb^N-\yb^{N-1})\rangle\\
&\quad + \tsum_{k=2}^N\hat\theta_k\alpha_k\langle \Lb(\xb^{k-1}-\xb^{k-2}),m(\yb^{k}-\yb^{k-1})-(m-1)(\yb^{k-1}-\yb^{k-2})\rangle\\
&=m\hat\theta_N\langle \Lb(\xb^N-\xb^{N-1}),\yb-\yb^N\rangle + \tsum_{k=2}^Nm\hat\theta_k\alpha_k\langle\Lb(\xb^{k-1}-\xb^{k-2}),\yb^{k}-\yb^{k-1}\rangle\\
&\quad -\tsum_{k=1}^N(m-1)\hat\theta_{k+1}\alpha_{k+1}\langle\Lb(\xb^{k}-\xb^{k-1}),\yb^{k}-\yb^{k-1}\rangle,
\end{align*}
where the third equality follows from \eqref{alpha_theta}, \eqref{eqn:algo4-i} and the fact that $\xb^{-1}=\xb^{0}$, and the last equality follows from \eqref{alpha_theta} and rearranging the terms. 
Also note that
\begin{align*}
\tsum_{k=1}^N&m\hat\theta_k\eta_k[\mathbf{V}(\xb^{k-1},\xb)-\mathbf{V}(\xb^k,\xb)-V_{j_k}(x_{j_k}^{k-1},x_{j_k}^k)]\\
&=m\hat\theta_1\eta_1\mathbf{V}(\xb^0,\xb)+\tsum_{k=2}^{N}(m\hat\theta_{k}\eta_{k}-m\hat\theta_{k-1}\eta_{k-1})\mathbf{V}(\xb^{k-1},\xb)-m\hat\theta_N\eta_N\mathbf{V}(\xb^N,\xb)\\
&\quad-\tsum_{k=1}^Nm\hat\theta_k\eta_k V_{j_k}(x_{j_k}^{k-1},x_{j_k}^k)\\
&\le m\hat\theta_1\eta_1\mathbf{V}(\xb^0,\xb)-m\hat\theta_N\eta_N\mathbf{V}(\xb^N,\xb) - \tsum_{k=1}^Nm\hat\theta_k\eta_kV_{j_k}(x_{j_k}^{k-1},x_{j_k}^k),
\end{align*}
where the last inequality follows from \eqref{theta_eta}. Similarly, by \eqref{theta_tau} we have 
\begin{align*}
\tsum_{k=1}^N&\tfrac{m\hat\theta_k\tau_k}{2}[\|\yb-\yb^{k-1}\|^2-\|\yb-\yb^k\|^2-\|y^{k-1}_{i_k}-y^{k}_{i_k}\|^2]\\
&\le \tfrac{m\hat\theta_1\tau_1}{2}\{\|\yb-\yb^0\|^2 -
\|\yb-\yb^N\|^2\}-\tsum_{k=1}^N\tfrac{m\hat\theta_k\tau_k}{2}\|y^{k-1}_{i_k}-y^k_{i_k}\|^2.
\end{align*}
Combining the above three results, we conclude that 
\begin{align*}
\tsum_{k=1}^N\hat\theta_k\Delta_k
&\le m\hat\theta_N\langle \Lb(\xb^N-\xb^{N-1}),\yb-\yb^N\rangle + \tsum_{k=2}^Nm\hat\theta_k\alpha_k\langle\Lb(\xb^{k-1}-\xb^{k-2}),\yb^{k}-\yb^{k-1}\rangle\\
&\quad -\tsum_{k=1}^N(m-1)\hat\theta_{k+1}\alpha_{k+1}\langle\Lb(\xb^{k}-\xb^{k-1}),\yb^{k}-\yb^{k-1}\rangle\\
&\quad+ m\hat\theta_1\eta_1\mathbf{V}(\xb^0,\xb)-m\hat\theta_N\eta_N\mathbf{V}(\xb^N,\xb) - \tsum_{k=1}^Nm\hat\theta_k\eta_kV_{j_k}(x_{j_k}^{k-1},x_{j_k}^k)\\
&\quad + \tfrac{m\hat\theta_1\tau_1}{2}\|\yb-\yb^0\|^2 -\tfrac{m\hat\theta_N\tau_N}{2}\|\yb-\yb^N\|^2-\tsum_{k=1}^N\tfrac{m\hat\theta_k\tau_k}{2}\|y_{i_k}^{k-1}-y_{i_k}^k\|^2\\
&\le m\hat\theta_N\langle \Lb(\xb^N-\xb^{N-1}),\yb-\yb^N\rangle - \tfrac{m\hat\theta_N\eta_N}{4}\|x_{j_N}^{N-1}-x_{j_N}^N\|^2\\
&\quad + \sum_{k=2}^N\left\{m\hat\theta_k\alpha_k {\cal L}_{i_{k},j_{k-1}}\langle x_{j_{k-1}}^{k-1}-x_{j_{k-1}}^{k-2},y_{i_k}^{k}-y_{i_k}^{k-1}\rangle -\tfrac{m\hat\theta_{k-1}\eta_{k-1}}{4}\|x_{j_{k-1}}^{k-1}-x_{j_{k-1}}^{k-2}\|^2 \right\}\\
&\quad + \sum_{k=1}^N\left\{(m-1)\hat\theta_{k+1}\alpha_{k+1}{\cal L}_{i_k,j_k}\langle x_{j_k}^{k}-x_{j_k}^{k-1},y_{i_k}^{k}-y_{i_k}^{k-1}\rangle - \tfrac{m\hat\theta_k\eta_k}{4}\|x_{j_k}^{k-1}-x_{j_k}^k\|^2 \right\}\\
&+ m\hat\theta_1\eta_1\mathbf{V}(\xb^0,\xb) + \tfrac{m\hat\theta_1\tau_1}{2}\{\|\yb-\yb^0\|^2 -\|\yb-\yb^N\|^2\} - \tsum_{k=1}^N\tfrac{m\hat\theta_k\tau_k}{2}\|y_{i_k}^{k-1}-y_{i_k}^k\|^2.
\end{align*}
Note that by \eqref{eta_tau_alpha1} and the fact that $b\langle u,v\rangle - a\|v\|^2/2 \le b^2\|u\|^2/(2a), \forall a>0$, for all $k\ge 2$, we have 
\begin{align*}
&m\hat\theta_k\alpha_k {\cal L}_{i_{k},j_{k-1}}\langle x_{j_{k-1}}^{k-1}-x_{j_{k-1}}^{k-2},y_{i_k}^{k}-y_{i_k}^{k-1}\rangle -\tfrac{m\hat\theta_{k-1}\eta_{k-1}}{4}\|x_{j_{k-1}}^{k-1}-x_{j_{k-1}}^{k-2}\|^2- \tfrac{m\hat\theta_k\tau_k}{4}\|y_{i_k}^{k-1}-y_{i_k}^k\|^2\\
&\le m\left(\tfrac{\hat\theta_k^2\alpha_k^2{\cal L}_{i_k,j_{k-1}}^2}{\hat\theta_{k-1}\eta_{k-1}}-\tfrac{\hat\theta_k\tau_k}{4}\right)\|y_{i_k}^{k-1}-y_{i_k}^k\|^2\le 0.
\end{align*}
Similarly, by \eqref{eta_tau_alpha2} for all $k\ge 1$, we have
\begin{align*}
&(m-1)\hat\theta_{k+1}\alpha_{k+1}{\cal L}_{i_k,i_k}\langle x_{j_k}^{k}-x_{j_k}^{k-1},y_{i_k}^{k}-y_{i_k}^{k-1}\rangle - \tfrac{m\hat\theta_k\eta_k}{4}\|x_{j_k}^{k-1}-x_{j_k}^k\|^2 - \tfrac{m\hat\theta_k\tau_k}{4}\|y_{i_k}^{k-1}-y_{i_k}^k\|^2 \le 0.
\end{align*}
Hence, combining the above three inequalities, we conclude that
\begin{align}
\tsum_{k=1}^N\hat\theta_k\Delta_k
&\le m\hat\theta_N\langle \Lb(\xb^N-\xb^{N-1}),\yb-\yb^N\rangle - \tfrac{m\hat\theta_N\eta_N}{4}\|x_{j_N}^{N-1}-x_{j_N}^N\|^2\nn \\
&\quad + m\hat\theta_1\eta_1\mathbf{V}(\xb^0,\xb) + \tfrac{m\hat\theta_1\tau_1}{2}\left\{\|\yb-\yb^0\|^2-\|\yb-\yb^N\|^2\right\}\label{rel2}\\
&\le m\hat\theta_N\langle\Lb(\xb^{N-1}-\xb^{N}),\yb^N\rangle- \tfrac{m\hat\theta_N\eta_N}{4}\|x_{j_N}^{N-1}-x_{j_N}^N\|^2 - \tfrac{m\hat\theta_N\tau_N}{2}\|\yb^N\|^2\nn\\
&\quad + m\hat\theta_1\eta_1\mathbf{V}(\xb^0,\xb) + \tfrac{\hat\theta_1\tau_1}{2}\|\yb^0\|^2 + m\langle\hat\theta_N\Lb(\xb^N-\xb^{N-1})+\hat\theta_1\tau_1(\yb^N-\yb^0),\yb\rangle\nn\\
&\le m\hat\theta_N\tsum_{i=1}^m\left(\tfrac{{\cal L}_{i,j_N}^2}{\eta_N}-\tfrac{\tau_N}{2}\right)\|y_i^N\|^2 + m\langle\hat\theta_N\Lb(\xb^N-\xb^{N-1})+\hat\theta_1\tau_1\yb^N,\yb\rangle\nn\\
&\quad + m\hat\theta_1\eta_1\mathbf{V}(\xb^0,\xb), 
\nn
\end{align}
where the second inequality follows from \eqref{eqn:def_V_i_s} and the fact that $b\langle u,v\rangle - a\|v\|^2/2 \le b^2\|u\|^2/(2a), \forall a>0$, and the last inequality also follows from the fact and $\yb^0=\0b$. In view of \eqref{eta_tau_alpha2} and \eqref{rel1}, we obtain
\begin{align}\label{Q_standard}
\bbe_{[i_k,j_k]}\left\{Q(\bar \zb^N,\zb)\right\}
&\le \hat\theta_0Q_0(\xb,\yb)  + m\hat\theta_1\eta_1\mathbf{V}(\xb^0,\xb) 
\nn\\
&\quad + \bbe_{[i_k,j_k]}\left\{m\langle\hat\theta_N\Lb(\xb^N-\xb^{N-1})+\hat\theta_1\tau_1\yb^N,\yb\rangle\right\}\nn\\
&=\hat\theta_0(F(\xb^0)-F(\xb)) + m\hat\theta_1\eta_1\mathbf{V}(\xb^0,\xb) 
\nn\\
&\quad + \bbe_{[i_k,j_k]}\left\{\langle \hat\theta_0 \Lb\xb^0+ m\hat\theta_N\Lb(\xb^N-\xb^{N-1})+m\hat\theta_1\tau_1\yb^N,\yb\rangle\right\},
\end{align}
where the last equality follows from the definition of $Q_0$ in \eqref{def_Q0}. The result in \eqref{bnd_Q1} immediately follows from the above relation.
Furthermore, from \eqref{rel1}, \eqref{rel2}, \eqref{theta_tau} and the facts that $Q(\bar\zb^N,\zb^*)\ge 0, \yb^0=\0b$, 
we have
\begin{align*}
0\le \bbe_{[i_k,j_k]}\left\{Q(\bar\zb^N,\zb^*)\right\}&\le \hat\theta_0Q_0(\xb^*,\yb^*) + m\hat\theta_1\eta_1\mathbf{V}(\xb^0,\xb^*) + \tfrac{m\hat\theta_1\tau_1}{2}\|\yb^*\|^2\\ 
&\quad + \bbe_{[i_k,j_k]}\left\{m\hat\theta_N\langle \Lb(\xb^N-\xb^{N-1}),\yb^*-\yb^N\rangle - \tfrac{m\hat\theta_N\eta_N}{4}\|x_{j_N}^{N-1}-x_{j_N}^N\|^2\right\}\nn \\
&\quad - \bbe_{[i_k,j_k]}\tfrac{m\hat\theta_N\tau_N}{2}\left\{\|\yb^*-\yb^N\|^2\right\}\\
&\le \hat\theta_0Q_0(\xb^*,\yb^*) + m\hat\theta_1\eta_1\mathbf{V}(\xb^0,\xb^*) + \tfrac{m\hat\theta_1\tau_1}{2}\|\yb^*\|^2\\ 
&\quad + \bbe_{[i_k,j_k]}\left\{\tsum_{i=1}^m m\hat\theta_N{\cal L}_{i,j_N}\langle x_{j_N}^N-x_{j_N}^{N-1},y_i^*-y_i^N\rangle\right\}\nn\\
&\quad - \bbe_{[i_k,j_k]}\left\{\tfrac{m\hat\theta_N\tau_N}{2}\tsum_{i=1}^m\|y_i^*-y_i^N\|^2
 + \tfrac{m\hat\theta_N\eta_N}{4}\|x_{j_N}^{N-1}-x_{j_N}^N\|^2\right\}, 
\end{align*}
which together with \eqref{def_Q0} and the fact that $b\langle u,v\rangle - a\|v\|^2/2 \le b^2\|u\|^2/(2a), \forall a>0$ imply that
\begin{align*}
\tfrac{m\hat\theta_N\eta_N}{4}\bbe_{[i_k,j_k]}\|x_{j_N}^{N-1}-x_{j_N}^N\|^2
&\le \hat\theta_0(F(\xb^0)-F(\xb^*) + \langle\Lb\xb^0,\yb^*\rangle) + m\hat\theta_1\eta_1\mathbf{V}(\xb^0,\xb^*) + \tfrac{m\hat\theta_1\tau_1}{2}\|\yb^*\|^2\\
&\quad + \bbe_{[i_k,j_k]}\{\tfrac{m\hat\theta_Nd_{max}^2}{\tau_N}\|x_{j_N}^{N-1}-x_{j_N}^{N}\|^2\},
\end{align*}
where 
the last inequality follows from the definition of $L$ in \eqref{def_Laplacian}. Similarly, we obtain
\begin{align*}
\tfrac{m\hat\theta_N\tau_N}{2}\bbe_{[i_k,j_k]}\|\yb^*-\yb^N\|^2 
&\le \hat\theta_0(F(\xb^0)-F(\xb^*) + \langle\Lb\xb^0,\yb^*\rangle) + m\hat\theta_1\eta_1\mathbf{V}(\xb^0,\xb^*) + \tfrac{m\hat\theta_1\tau_1}{2}\|\yb^*\|^2\\
&\quad + \bbe_{[i_k,j_k]}\{\tfrac{m\hat\theta_Nd_{max}^2}{\eta_N}\|\yb^*-\yb^N\|^2\},
\end{align*}
which implies the result in \eqref{bnd_xyn}.
\end{proof}

\noindent{\bf Proof of Theorem~\ref{main_adpd}.}
Let us set $\{\hat\theta_k\}$ as follow
\beq\label{def_htheta}
\hat\theta_k = 
\begin{cases}
\tfrac{m}{N+m}, \ &k = 0,\\
\tfrac{1}{N+m}, \ &k = 1,\dots,N.
\end{cases}
\eeq
Therefore, it is easy to check that \eqref{para_adpd} satisfies conditions \eqref{theta_eta}-\eqref{eta_tau_alpha2}. Also note that by \eqref{def_xweight}, we have 
\beq \label{def_theta}
\theta_k=
  \begin{cases}
  \tfrac{1}{N+m}, \ &k=1,\dots,N-1,\\
  \tfrac{m}{N+m}, \ &k=N,
  \end{cases}
\eeq
which implies that $\bar\xb^N=\tfrac{1}{N+m}(\tsum_{k=0}^{N-1}\underline \xb^k+m\underline \xb^N)$.
By plugging the parameter setting in \eqref{bnd_Q1}, we have
\begin{align}\label{bnd_Qv1}
  \bbe_{[i_k,j_k]}\{Q(\bar\zb^N;\xb^*,\yb)\}\le \tfrac{m}{N+m}\left[F(\xb^0)-F(\xb^*) + 2md_{max}\mathbf{V}(\xb^0,\xb^*)\right]+ \bbe_{[i_k,j_k]}\{\langle \vb,\yb\rangle\}.
\end{align}
Observe that from \eqref{def_resi} and \eqref{para_adpd},
\begin{align*}
\bbe_{[i_k,j_k]}\{\|\vb\|\}&\le \tfrac{m}{N+m}\bbe_{[i_k,j_k]}\left[\|\Lb\xb^0\| + 2d_{max}\|x_{j_N}^N-x_{j_N}^{N-1}\|+2md_{max}(\|\yb^*-\yb^N\| + \|\yb^*\|)\right].
\end{align*}
By \eqref{bnd_xyn}, \eqref{para_adpd} and Jensen's inequality, we have
\begin{align*}
(\bbe\{\|x_{j_N}^N-x_{j_N}^{N-1}\|\})^2&\le \bbe\{\|x_{j_N}^N-x_{j_N}^{N-1}\|^2\}\le 4\left[\tfrac{F(\xb^0)-F(\xb^*) + \la \Lb\xb^0,\yb^*\ra}{md_{max}} + 2\mathbf{V}(\xb^0,\xb^*) + \|\yb^*\|^2\right], 
\\
(\bbe\{\|\yb^*-\yb^N\|\})^2&\le \bbe\{\|\yb^*-\yb^N\|^2\}\le 2\left[\tfrac{F(\xb^0)-F(\xb^*) + \la \Lb\xb^0,\yb^*\ra}{md_{max}} + 2\mathbf{V}(\xb^0,\xb^*) + \|\yb^*\|^2\right].
\end{align*}
Hence, in view of the above three inequalities, we conclude that
\begin{align*}
\bbe_{[i_k,j_k]}\{\|\vb\|\}&\le \tfrac{m}{N+m}\Big\{\|\Lb\xb^0\|+ 2md_{max}\|\yb^*\|\\
&\quad + 7md_{max}\sqrt{\tfrac{F(\xb^0)-F(\xb^*) + \la \Lb\xb^0,\yb^*\ra}{md_{max}} + 2\mathbf{V}(\xb^0,\xb^*) + \|\yb^*\|^2}\\
&= {\cal O}\left\{\tfrac{m}{N+m}\left[\|\Lb\xb^0\|+ md_{max}\|\yb^*\| + md_{max}\sqrt{\tfrac{F(\xb^0)-F(\xb^*) + \la \Lb\xb^0,\yb^*\ra}{md_{max}} + \mathbf{V}(\xb^0,\xb^*)}\right]\right\}.
\end{align*}
Furthermore, by \eqref{bnd_Qv1} we have
\[
  \bbe_{[i_k,j_k]}\{g(\vb,\bar\zb^N)\}\le \tfrac{m}{N+m}\left[F(\xb^0)-F(\xb^*) + 2md_{max}\mathbf{V}(\xb^0,\xb^*)\right].
\]
The results in \eqref{bnd_adpd} immediately follow from Proposition~\ref{prop:approx} and the above two inequalities.
\hfill $\square$


\subsection{Convergence properties of Algorithm~\ref{alg:AASDCS}}
Before we provide the proof for Theorem~\ref{main_aasdcs}, which establishes the main convergence results for AA-SDCS, we state in the following proposition a general result for the \inner~procedure. 
For notation convenience, we use the notations defined the in \inner~procedure (cf. Algorithm~\ref{alg:AASDCS}) 
and let 
\beq\label{def_Lambda}
\Lambda_t: = 
\begin{cases}
1, \ &t = 1,\\
(1-\lambda_t)\Lambda_t, \ &t\ge2.
\end{cases}
\eeq



\begin{proposition}\label{prop:inner}
  If $\{\beta_t\}$ and $\{\lambda_t\}$ in the \inner~procedure satisfy
  \begin{align}
  \lambda_1 &= 1, \label{lam_1}\\
  \mu + \eta + \beta_t&>(\mathcal{C}+L)\lambda_t^2, \ t = 1,\dots,T, \label{mu_L}\\
  \tfrac{\beta_t}{\Lambda_t} &= \tfrac{\beta_{t-1}}{\Lambda_{t-1}}, \ t = 1,\dots,T, \label{beta_Lambda}
  \end{align}
  then, under assumptions \eqref{assume:unbiased} and \eqref{assume:sm_bounded}, for $u \in U$,
  \begin{align}\label{inner_recursion}
  \bbe_{[\xi]}\Phi(\underline u^T) - \Phi(u) 
&\le \Lambda_T \beta_1 V(u^0,u) - (\Lambda_T\beta_1+\mu + \eta)\bbe_{\xi}V(u^T,u) + \Lambda_T\tsum_{t=1}^T\tfrac{2(M^2+\sigma^2)\lambda_t^2}{(\mu + \eta + \beta_t -(\mathcal{C}+L)\lambda_t^2)\Lambda_t},
  \end{align}
  where $\bbe_{[\xi]}$ represents taking the expectation over $\{\xi_i^1, \dots, \xi_i^{T}\}$ and $\Phi$ is defined as
\begin{align}\label{eqn:defphi}
\Phi(u) := \la w, u\ra + \phi(u) + \eta V(x,u).
\end{align}
\end{proposition}
\begin{proof}
Note that in view of \eqref{eqn:def_V}, \eqref{eqn:def_V_i_s} and \eqref{eqn:proxquad}, we have
\[
  \tfrac{1}{2}\|u_1 - u_2\|^2 \le V(x,u_1) - V(x,u_2) - \la \nabla V(x, u_2), u_1-u_2\ra = V(u_2, u_1) \le \tfrac{\mathcal{C}}{2}\|u_1-u_2\|^2, \ \forall u_1, u_2\in U, 
\]
where $\nabla V(x, u_2)$ denotes the gradient of $V(x, \cdot)$ w.r.t. $u_2$ for a given $x$, and the above result together with \eqref{eqn:nonsmooth} imply $\phi(\cdot)$ satisfies 
\[
\tfrac{\mu+ \eta}{2}\|u_1-u_2\|^2 \le \Phi(u_1) - \Phi(u_2) - \la \nabla \Phi(u_2), u_1-u_2\ra \le \tfrac{\mathcal{C}+L}{2}\|u_1-u_2\|^2 + M\|u_1-u_2\|, \ \forall u_1,u_2 \in U.
\]
Hence, by the proof of Theorem 1 in \cite{GhaLan10-1}, we can conclude that 
\begin{align*}
\bbe_{[\xi]}\Phi(\underline u^T) - \Phi(u) \le \Lambda_T\beta_1 V(u^0,u) - (\Lambda_T\beta_1 + \mu+\eta)\bbe_{[\xi]}V(u^T,u) + \Lambda_t\tsum_{t=1}^T \tfrac{2\lambda_t^2(M^2+\sigma^2)}{\Lambda_t(\mu + \eta+\beta_t-(\mathcal{C}+L)\lambda_t^2)}.
\end{align*}
\end{proof}

We are now ready to present the main convergence property of the AA-SDCS method stated in Algorithm~\ref{alg:AASDCS} when the objective functions $f_i, i=1,\dots,m$, are general convex.
\begin{proposition}\label{prop_bndQ_aasdcs}
Let the iterates $(\underline \xb^k, \xb^{k})$ and $\hat\yb^k$, $k=1,\dots,N$, be generated by Algorithm~\ref{alg:AASDCS} and be defined as in \eqref{def_yhat}, respectively, and let $\bar \zb^N:=(\tsum_{k=0}^N\theta_k\underline \xb^k,\tsum_{k=0}^N\hat\theta_k\hat\yb^k)$. 
Assume that the objective $f_i, i = 1,\dots,m$, are general convex functions, i.e., $\mu=0, \ L,M\ge 0$ in \eqref{eqn:nonsmooth}. 
Let the parameters $\{\alpha_k\}$, $\{\tau_k\}$, and $\{\eta_k\}$ in Algorithm~\ref{alg:AASDCS} satisfy \eqref{theta_tau} and 
\begin{align}
\hat\theta_k\left(\tfrac{\mathcal{C} + L}{T_k(T_k+1)}+\eta_k\right)&\le \hat\theta_{k-1}\left(\tfrac{\mathcal C +L}{T_{k-1}(T_{k-1}+1)}+\eta_{k-1}\right), \ k=2,\dots,N, \label{theta_Tk_eta}\\
\alpha_k\hat\theta_k &=\hat\theta_{k-1}, \ k=2,\dots,N, \label{alpha_htheta}\\
8m\alpha_kd^2_{max}&\le \eta_{k-1}\tau_{k}, k =2, \dots,N,\label{alpha_d_eta_tau}\\
8(m-1)^2d^2_{max}&\le m\eta_{k}\tau_{k}, k = 1,\dots,N, \label{m_d_eta_tau}
\end{align}
where $\{\hat \theta_k\}$ is some given weight sequence.
Let the parameters $\{\lambda_t\}$ and $\{\beta_t\}$ in the \inner~procedure of Algorithm~\ref{alg:AASDCS} be set to \eqref{para_sgd1}.
Then, for any $\zb:=(\xb,\yb)\in X^m\times\bbr^{md}$, we have
\beq\label{bnd_Q_aasdcs}
\bbe\{Q(\bar\zb^N;\zb)\}\le \hat\theta_0(F(\xb^0)-F(\xb)) + m\hat\theta_1\left(\tfrac{4(\mathcal C+L)}{T_1(T_1+1)}+\eta_1\right)\mathbf{V}(\xb^0,\xb)+ \bbe\{\langle \vb,\yb\rangle\}+ \tsum_{k=1}^{N}\tfrac{8m(M^2+\sigma^2)\hat\theta_k}{\eta_k(T_k+1)},
\eeq
where $\bbe$ represents taking the expectation over all random variables, $Q$ is defined in \eqref{eqn:gap} and $\vb$ are defined as 
\beq\label{def_resi_aasdcs}
\vb:=\hat\theta_0 \Lb\xb^0+ \hat\theta_N\Lb(\hat{\underline\xb}^N-\xb^{N-1})+m\hat\theta_1\tau_1\yb^N.
\eeq
Furthermore, for any saddle point $(\xb^*,\yb^*)$ of \eqref{eqn:saddle}, we have
\begin{align}\label{bnd_xyn_aasdcs}
\tfrac{\hat\theta_N}{4}&\left(1-\tfrac{2\|\Lb\|^2}{m\eta_N\tau_N}\right)\max\left\{\eta_N\bbe\|\hat{\underline \xb}^N-\xb^{N-1}\|^2,2m\tau_N\bbe\|\yb^*-\yb^N\|^2\right\}\nn\\
&\le \hat\theta_0(F(\xb^0)-F(\xb^*) + \langle\Lb\xb^0,\yb^*\rangle) + m\hat\theta_1\left(\tfrac{4(\mathcal C+L)}{T_1(T_1+1)}+\eta_1\right)\mathbf{V}(\xb^0,\xb^*)\nn\\
&\quad + \tfrac{m\hat\theta_1\tau_1}{2}\|\yb^*\|^2 + \tsum_{k=1}^N\tfrac{8m(M^2+\sigma^2)\hat\theta_k}{(T_k+1)\eta_k}.
\end{align}
\end{proposition}
\begin{proof}
Since $f_i$'s are general convex function, we have $\mu = 0$ and $L,M\ge0$ (cf. \eqref{eqn:nonsmooth}). 
Also note that $\lambda_t$ and $\beta_t$ defined in \eqref{para_sgd1} satisfy condition \eqref{lam_1}-\eqref{beta_Lambda}.
Therefore, substituting $\phi := f_i$, and $\lambda_t$ and $\beta_t$, relation \eqref{inner_recursion} can be rewritten as the following,\footnote{We added the subscript $i$ to emphasize that this inequality holds for any agent $i \in \Vc$ with $\phi = f_i$. More specifically, $\Phi_i(u_i) := \la w_i, u_i\ra + f_i(u_i) + \eta V_i(x_i,u_i)$.}
  \begin{align*}
  \bbe_{[\xi]}\Phi_i(\underline u_i^T) - \Phi_i(u_i) 
&\le \Lambda_T \beta_1 V_i(u_i^0,u_i) - (\Lambda_T\beta_1+ \eta)\bbe_{\xi}V_i(u_i^T,u_i) + \Lambda_T\tsum_{t=1}^T\tfrac{2(M^2+\sigma^2)\lambda_t^2}{(\eta + \beta_t - (\mathcal{C}+L)\lambda_t^2)\Lambda_t},
  \end{align*}
  Summing up the above inequality from $i\in [m]$, and using the definitions of $\hat \xb^k_{+}$ and $\hat{\underline \xb}^k$ in \eqref{def_xhatu}, we obtain 
  \begin{align*}
  \bbe_{[\xi]}\Phi^k(\hat{\underline \xb}^k) - \Phi^k(\xb) 
&\le \Lambda_{T_k} \beta_1 \mathbf{V}(\xb^{k-1},\xb) - (\Lambda_{T_k}\beta_1+ \eta_k)\bbe_{\xi}\mathbf{V}(\xb_+^k,\xb) + \Lambda_{T_k}\tsum_{t=1}^{T_k}\tfrac{2m(M^2+\sigma^2)\lambda_t^2}{(\eta_k + \beta_t - (\mathcal{C}+L)\lambda_t^2)\Lambda_t},
  \end{align*}
  where $\Phi^k(\xb)= \la \Lb\xb, \tilde\yb^k\ra + F(\xb) + \eta_k\mathbf{V}(\xb^{k-1},\xb)$. By plugging into the above relation the values of $\lambda_t$ and $\beta_t$ in \eqref{para_sgd1}, together with the definition of $\Phi^k(\xb)$ and rearranging the terms, we have $\forall \xb\in X^m$
  \begin{align}\label{condi_sgd}
  \bbe_{[\xi]}\left\{\la \Lb(\hat {\underline \xb}^k-\xb), \tilde \yb^k\ra + F(\hat {\underline \xb}^k)- F(\xb)\right\}
  &\le \left(\tfrac{4(\mathcal{C}+L)}{T_k(T_k+1)}+\eta_k\right)\bbe_{[\xi]}\left[\mathbf{V}(\xb^{k-1},\xb) - \mathbf{V}(\hat\xb^k_{+},\xb)\right]\nn\\
  &\quad-\eta_k\bbe_{[\xi]}\{\mathbf{V}(\xb^{k-1},\hat {\underline\xb}^k)\}+\tfrac{8m(M^2+\sigma^2)}{(T_k+1)\eta_k}.
  \end{align}
  By the definitions of $Q$ in \eqref{eqn:gap} and $\bar \zb^N$, and the convexity of $F(\cdot)$, we have
  \begin{align*}
Q(\bar \zb^N;\zb)
&\le \tsum_{k=0}^N\theta_k[F(\underline \xb^k)-F(\xb)+\langle \Lb\underline \xb^k,\yb\rangle]
-\tsum_{k=0}^N\hat\theta_k\langle\Lb\xb,\hat \yb^k\rangle.
\end{align*}
Taking expectation over $i_1,j_1,\dots,i_k,j_k$ and applying Lemma~\ref{lem_exp}, we obtain
\begin{align*}
\bbe_{[i_k,j_k]}\left\{Q(\bar \zb^N;\zb)\right\}&\le\bbe_{[i_k,j_k]}\left\{\tsum_{k=0}^N\hat \theta_k[F(\hat{\underline\xb}^k)-F(\xb)+\langle\Lb\hat{\underline\xb}^k,\yb\rangle - \langle \Lb\xb,\hat\yb^k\rangle]\right\}.
\end{align*}
Moreover, if we replace \eqref{condi_subp} by \eqref{condi_sgd} in Lemma~\ref{lem_recursion}, we can conclude the following result similar to \eqref{raw_bnd_Q}
\begin{align*}
\bbe\{Q(\bar\zb^N;\zb)\}
&\le \hat\theta_0Q_0(\xb,\yb) + \tsum_{k=1}^N\tfrac{8m(M^2+\sigma^2)\hat\theta_k}{(T_k+1)\eta_k}\\
&\quad + \bbe\tsum_{k=1}^N\left\{\hat\theta_k\langle\Lb(\hat{\underline \xb}^k-\tilde\xb^k),\yb-\tilde\yb^k\rangle -\hat\theta_k\eta_k\mathbf{V}(\xb^{k-1},\hat{\underline \xb}^k)]\right\}\nn\\
&\quad + \bbe\left\{\tsum_{k=1}^Nm\hat\theta_k\left(\tfrac{4(\mathcal{C}+L)}{T_k(T_k+1)}+\eta_k\right)[\mathbf{V}(\xb^{k-1},\xb)-\mathbf{V}(\xb^k,\xb)]\right\}\nn\\
&\quad + \bbe\left\{\tsum_{k=1}^N\tfrac{m\hat\theta_k\tau_k}{2}[\|\yb-\yb^{k-1}\|^2-\|\yb-\yb^k\|^2-\|y_{i_k}^{k-1}-y_{i_k}^k\|^2]\right\},
\end{align*}
where $\bbe$ represents taking the expectation over all random variables.
Therefore, we have
\begin{align}\label{bnd_Q_aasdcs_con}
\bbe\{Q(\bar\zb^N;\zb)\}&\le \hat\theta_0Q_0(\xb,\yb) + \tsum_{k=1}^N\tfrac{8m(M^2+\sigma^2)\hat\theta_k}{(T_k+1)\eta_k} + \bbe\left\{\tsum_{k=1}^N\hat\theta_k\tilde \Delta_k\right\},
\end{align}
where 
\begin{align}\label{def_tDelta}
\tilde\Delta_k:= &\langle \Lb(\hat{\underline \xb}^k-\tilde\xb^k),\yb-\tilde\yb^{k}\rangle 
+ m\left(\tfrac{4(\mathcal{C}+L)}{T_k(T_k+1)}+\eta_k\right)[\mathbf{V}(\xb^{k-1},\xb)-\mathbf{V}(\xb^k,\xb)]\nn\\
&\quad -\eta_k\mathbf{V}(\xb^{k-1},\hat{\underline \xb}^k)
 + \tfrac{m\tau_k}{2}[\|\yb-\yb^{k-1}\|^2-\|\yb-\yb^k\|^2-\|y_{i_k}^{k-1}-y_{i_k}^k\|^2].
\end{align}
We now provide a bound for $\bbe\{\tsum_{k=1}^N\hat\theta_k\tilde \Delta_k\}$. Observe that $\tilde \Delta_k$ is different from $\Delta_k$ defined in \eqref{def_Delta} in first three terms, however, they can be bounded via the same technique. Note that by \eqref{eqn:algo1-acs}, we obtain   
\begin{align*}
\bbe&\left\{\tsum_{k=1}^N\hat\theta_k\langle \Lb(\hat{\underline\xb}^k-\tilde \xb^k),\yb-\tilde\yb^k\rangle\right\}\\
&= \bbe\left\{\tsum_{k=1}^N\hat\theta_k\langle \Lb((\hat{\underline\xb}^k-\xb^{k-1})-\alpha_k(m\underline\xb^{k-1}-(m-1)\underline \xb^{k-2}-\xb^{k-2}),\yb-\tilde \yb^k\rangle\right\}\\
&= \bbe\left\{\tsum_{k=1}^N\left[\hat\theta_k\langle\Lb(\hat{\underline \xb}^k-\xb^{k-1}),\yb-\tilde\yb^k\rangle - \hat\theta_k\alpha_k\langle \Lb(\hat{\underline\xb}^{k-1}-\xb^{k-2}),\yb-\tilde\yb^{k-1}\rangle\right]\right\}\\
&\quad+ \bbe\left\{\tsum_{k=1}^N\hat\theta_k\alpha_k\langle \Lb(
\hat{\underline\xb}^{k-1}-\xb^{k-2}),\tilde\yb^k-\tilde \yb^{k-1})\rangle\right\}\\
&\varstackrel{\eqref{alpha_htheta},\eqref{eqn:algo4-acs}}{=}\hat\theta_N\langle \Lb(\hat{\underline \xb}^N-\xb^{N-1}),\yb-\yb^N-(m-1)(\yb^N-\yb^{N-1})\rangle\\
&\quad + \tsum_{k=2}^N\hat\theta_k\alpha_k\langle \Lb(\hat{\underline \xb}^{k-1}-\xb^{k-2}),m(\yb^{k}-\yb^{k-1})-(m-1)(\yb^{k-1}-\yb^{k-2})\rangle\\
&\varstackrel{\eqref{alpha_htheta}}{=}\hat\theta_N\langle \Lb(\hat{\underline \xb}^N-\xb^{N-1}),\yb-\yb^N\rangle + \tsum_{k=2}^Nm\hat\theta_k\alpha_k\langle\Lb(\hat{\underline \xb}^{k-1}-\xb^{k-2}),\yb^{k}-\yb^{k-1}\rangle\\
&\quad -\tsum_{k=1}^N(m-1)\hat\theta_{k+1}\alpha_{k+1}\langle\Lb(\hat{\underline \xb}^{k}-\xb^{k-1}),\yb^{k}-\yb^{k-1}\rangle,
\end{align*}
which together with \eqref{theta_Tk_eta} and \eqref{theta_tau} imply that
\begin{align*}
\bbe\{\tsum_{k=1}^N\hat\theta_k\tilde \Delta_k\}
&\le \bbe\left\{\hat\theta_N\langle \Lb(\hat{\underline \xb}^N-\xb^{N-1}),\yb-\yb^N\rangle + \tsum_{k=2}^Nm\hat\theta_k\alpha_k\langle\Lb(\hat{\underline \xb}^{k-1}-\xb^{k-2}),\yb^{k}-\yb^{k-1}\rangle\right\}\\
&\quad -\bbe\left\{\tsum_{k=1}^N(m-1)\hat\theta_{k+1}\alpha_{k+1}\langle\Lb(\hat{\underline \xb}^{k}-\xb^{k-1}),\yb^{k}-\yb^{k-1}\rangle + \tsum_{k=1}^N\hat\theta_k\eta_k\mathbf{V}(\xb^{k-1},\hat{\underline \xb}^k)\right\}\\
&\quad+ \bbe\left\{m\hat\theta_1\left(\tfrac{4(\mathcal C+L)}{T_1(T_1+1)}+\eta_1\right)\mathbf{V}(\xb^0,\xb)-m\hat\theta_N\left(\tfrac{4(\mathcal C+L)}{T_N(T_N+1)}+\eta_N\right)\mathbf{V}(\xb^N,\xb)\right\} \\
&\quad + \bbe\left\{\tfrac{m\hat\theta_1\tau_1}{2}\|\yb-\yb^0\|^2 -\tfrac{m\hat\theta_N\tau_N}{2}\|\yb-\yb^N\|^2-\tsum_{k=1}^N\tfrac{m\hat\theta_k\tau_k}{2}\|y_{i_k}^{k-1}-y_{i_k}^k\|^2\right\}\\
&\le \bbe\left\{\hat\theta_N\la \Lb(\hat{\underline \xb}^N-\xb^{N-1}), \yb - \yb^N\ra -\tfrac{\hat\theta_N\eta_N}{4}\|\xb^{N-1}-\hat{\underline \xb}^N\|^2-\tsum_{k=1}^N\tfrac{m\hat\theta_k\tau_k}{2}\|y_{i_k}^{k-1}-y_{i_k}^k\|^2\right\}\\
&\quad + \tsum_{k=2}^N\bbe\left\{m\hat\theta_k\alpha_k\la \Lb(\hat{\underline \xb}^{k-1}-\xb^{k-2}), \yb^{k}-\yb^{k-1}\ra -\tfrac{\hat\theta_{k-1}\eta_{k-1}}{4}\|\hat{\underline \xb}^{k-1} - \xb^{k-2}\|^2\right\}\\
&\quad + \tsum_{k=1}^N\bbe\left\{(m-1)\hat\theta_{k+1}\alpha_{k+1}\la \Lb(\hat{\underline \xb}^{k}- \xb^{k-1}),\yb^k-\yb^{k-1}\ra - \tfrac{\hat\theta_k\eta_k}{4}\|\xb^{k-1}-\hat{\underline \xb}^{k}\|^2\right\}\\
&\quad + m\hat\theta_1\left(\tfrac{4(\mathcal C+L)}{T_1(T_1+1)}+\eta_1\right)\mathbf{V}(\xb^0,\xb)+ \tfrac{m\hat\theta_1\tau_1}{2}\bbe\{\|\yb-\yb^0\|^2 -\|\yb-\yb^N\|^2\}.
\end{align*}
Noting that by the fact that $b\langle u,v\rangle - a\|v\|^2/2 \le b^2\|u\|^2/(2a), \forall a>0$ and \eqref{alpha_htheta} and \eqref{alpha_d_eta_tau}, for all $k\ge 2$, we have 
\begin{align*}
&\tsum_{j=1}^m\left\{m\hat\theta_k\alpha_k {\cal L}_{i_k,j}\langle \hat{\underline x}_j^{k-1}-x_j^{k-2},y_{i_k}^{k}-y_{i_k}^{k-1}\rangle -\tfrac{\hat\theta_{k-1}\eta_{k-1}}{4}\|\hat{\underline x}_j^{k-1}-x_j^{k-2}\|^2\right\}- \tfrac{m\hat\theta_k\tau_k}{4}\|y_{i_k}^{k-1}-y_{i_k}^k\|^2\\
&\le m\left(\tsum_{j=1}^m\tfrac{m\hat\theta_k^2\alpha_k^2{\cal L}_{i_k,j}^2}{\hat\theta_{k-1}\eta_{k-1}}-\tfrac{\hat\theta_k\tau_k}{4}\right)\|y_{i_k}^{k-1}-y_{i_k}^k\|^2\le 0.
\end{align*}
Similarly, by \eqref{m_d_eta_tau} for all $k\ge 1$, we have
\[
  \tsum_{j=1}^m\left\{(m-1)\hat\theta_{k+1}\alpha_{k+1}{\cal L}_{i_k,j}\la\hat{\underline x}_j^{k} - x_j^{k-1},y_{i_k}^k-y_{i_k}^{k-1}\ra-\tfrac{\hat\theta_k\eta_k}{4}\|x_j^{k-1}-\hat{\underline x}_j^{k}\|^2\right\}-\tfrac{m\hat\theta_k\tau_k}{4}\|y_{i_k}^{k-1}-y_{i_k}^k\|^2\le 0.
\]
Hence, in view of the above three results, we obtain
\begin{align}\label{rel5}
\bbe\left\{\tsum_{k=1}^N\hat\theta_k\tilde \Delta_k\right\}
&\le \hat\theta_N\bbe\{\langle \Lb(\hat{\underline\xb}^N-\xb^{N-1}),\yb-\yb^N\rangle\} - \tfrac{\hat\theta_N\eta_N}{4}\bbe\{\|\hat{\underline\xb}^N-\xb^{N-1}\|^2\}\nn \\
&\quad + m\hat\theta_1\left(\tfrac{4(\mathcal C+L)}{T_1(T_1+1)}+\eta_1\right)\mathbf{V}(\xb^0,\xb)+ \tfrac{m\hat\theta_1\tau_1}{2}\bbe\{\|\yb-\yb^0\|^2 -\|\yb-\yb^N\|^2\}.
\end{align}
Following the same procedure as we used in Proposition~\ref{prop_bndQ} (cf. \eqref{Q_standard}), and using the above result and \eqref{bnd_Q_aasdcs_con}, we can conclude that 
\begin{align*}
\bbe\left\{Q(\bar \zb^N,\zb)\right\}
&\le \hat\theta_0(F(\xb^0)-F(\xb)) + m\hat\theta_1\left(\tfrac{4(\mathcal C+L)}{T_1(T_1+1)}+\eta_1\right)\mathbf{V}(\xb^0,\xb)+ \tsum_{k=1}^N\tfrac{8m(M^2+\sigma^2)\hat\theta_k}{(T_k+1)\eta_k}
\nn\\
&\quad + \bbe\left\{\langle \hat\theta_0 \Lb\xb^0+ \hat\theta_N\Lb(\hat{\underline\xb}^N-\xb^{N-1})+m\hat\theta_1\tau_1\yb^N,\yb\rangle\right\},
\end{align*}
which implies the result in \eqref{bnd_Q_aasdcs}. Furthermore, from \eqref{bnd_Q_aasdcs_con}, \eqref{rel5}, \eqref{theta_tau}, and the fact that $Q(\bar \zb^N,\zb^*)\ge 0, \yb^0=\0b$, 
we have
\begin{align*}
0\le \bbe\{Q(\bar\zb^N,\zb^*)\}&\le \hat\theta_0Q_0(\xb^*,\yb^*) + m\hat\theta_1\left(\tfrac{4(\mathcal C+L)}{T_1(T_1+1)}+\eta_1\right)\mathbf{V}(\xb^0,\xb^*) + \tfrac{m\hat\theta_1\tau_1}{2}\|\yb^*\|^2\\
&\quad+\bbe\left\{\hat\theta_N\la \Lb(\hat{\underline \xb}^N-\xb^{N-1}),\yb^*-\yb^N\ra - \tfrac{\hat\theta_N\eta_N}{4}\|\hat{\underline \xb}^{N-1}-\xb^N\|^2\right\}\\
&\quad -\bbe\left\{\tfrac{m\hat\theta_N\tau_N}{2}\|\yb^*-\yb^N\|^2\right\} + \tsum_{k=1}^N\tfrac{8m(M^2+\sigma^2)\hat\theta_k}{(T_k+1)\eta_k},
\end{align*} 
which together with the fact that $b\langle u,v\rangle - a\|v\|^2/2 \le b^2\|u\|^2/(2a), \forall a>0$ and \eqref{def_Q0} imply that
\begin{align}\label{bnd_huxN}
\tfrac{\hat\theta_N\eta_N}{4}\bbe\|\hat{\underline \xb}^{N-1}-\xb^N\|^2
&\le \hat\theta_0(F(\xb^0)-F(\xb^*) + \langle\Lb\xb^0,\yb^*\rangle) + m\hat\theta_1\left(\tfrac{4(\mathcal C+L)}{T_1(T_1+1)}+\eta_1\right)\mathbf{V}(\xb^0,\xb^*)\nn\\ 
&\quad + \tfrac{m\hat\theta_1\tau_1}{2}\|\yb^*\|^2+ \bbe\left\{\tfrac{\hat\theta_N\|\Lb\|^2}{2m\tau_N}\|\hat{\underline \xb}^{N-1}-\xb^{N}\|^2\right\} + \tsum_{k=1}^N\tfrac{8m(M^2+\sigma^2)\hat\theta_k}{(T_k+1)\eta_k}.
\end{align}
Similarly, we can obtain
\begin{align*}
\tfrac{m\hat\theta_N\tau_N}{2}\bbe\|\yb^*-\yb^N\|^2 
&\le \hat\theta_0(F(\xb^0)-F(\xb^*) + \langle\Lb\xb^0,\yb^*\rangle) + m\hat\theta_1\left(\tfrac{4(\mathcal C+L)}{T_1(T_1+1)}+\eta_1\right)\mathbf{V}(\xb^0,\xb^*)\\
&\quad + \tfrac{m\hat\theta_1\tau_1}{2}\|\yb^*\|^2
 + \bbe\{\tfrac{m\hat\theta_N\|\Lb\|^2}{\eta_N}\|\yb^*-\yb^N\|^2\} + \tsum_{k=1}^N\tfrac{8m(M^2+\sigma^2)\hat\theta_k}{(T_k+1)\eta_k},
\end{align*}
which implies the result in \eqref{bnd_xyn_aasdcs}.
\end{proof}

{\bf Proof of Theorem~\ref{main_aasdcs}} Let us set $\{\hat\theta_k\}$ as \eqref{def_htheta}
Therefore, it is easy to check that parameter settings \eqref{para_sgd1} and \eqref{para_aasdcs} satisfies conditions \eqref{lam_1} - \eqref{beta_Lambda}, \eqref{theta_tau}, and \eqref{theta_Tk_eta} - \eqref{m_d_eta_tau}. Also note that by \eqref{def_xweight}, $\{\theta_k\}$ is given by \eqref{def_theta},
which implies that $\bar\xb^N=\tfrac{1}{N+m}(\tsum_{k=0}^{N-1}\underline \xb^k+m\underline \xb^N)$.
By plugging the parameter setting in \eqref{bnd_Q_aasdcs}, we have
\begin{align}\label{bnd_Qv2}
\bbe\{Q(\bar\zb^N;\xb^*,\yb)\}&\le \tfrac{m}{N+m}\left[F(\xb^0)-F(\xb^*) + 8md_{max}\mathbf{V}(\xb^0,\xb^*) + \tfrac{2\mathcal{D}}{m}\right]
+ \bbe\{\langle \vb,\yb\rangle\}.
\end{align} 
Observe that from \eqref{def_resi_aasdcs} and \eqref{para_aasdcs}
\begin{align*}
\bbe\{\|\vb\|\}&\le \tfrac{m}{N+m}\bbe\left[\|\Lb\xb^0\| + \tfrac{\|\Lb\|}{m}\|\hat{\underline \xb}^N-\xb^{N-1}\|+2d_{max}(\|\yb^*-\yb^N\| + \|\yb^*\|)\right].
\end{align*}
By \eqref{bnd_xyn_aasdcs}, \eqref{para_aasdcs}, and Jensen's inequality, we have
\begin{align*}
(\bbe\{\|\hat{\underline \xb}^N-\xb^{N-1}\|\})^2&\le \bbe\{\|\hat{\underline \xb}^N-\xb^{N-1}\|^2\}\\
&\le\tfrac{2(F(\xb^0)-F(\xb^*)+ \la \Lb\xb^0,\yb^*\ra)}{d_{max}} + 16m\mathbf{V}(\xb^0,\xb^*) + 2\|\yb^*\|^2+ \tfrac{4\mathcal{D}}{md_{max}},\\
(\bbe\{\|\yb^*-\yb^N\|\})^2&\le \bbe\{\|\yb^*-\yb^N\|^2\}\\
&\le \tfrac{2(F(\xb^0)-F(\xb^*)+\la \Lb\xb^0,\yb^*\ra)}{d_{max}} + 16m\mathbf{V}(\xb^0,\xb^*) + 2\|\yb^*\|^2+ \tfrac{4\mathcal{D}}{md_{max}}.
\end{align*}
Hence, in view of the above three inequalities, we obtain
\begin{align*}
\bbe\{\|\vb\|\}&\le \tfrac{m}{N+m}\Big\{\|\Lb\xb^0\|+ 2d_{max}\|\yb^*\|\\
&\quad + 3d_{max}\sqrt{\tfrac{2(F(\xb^0)-F(\xb^*)+\la \Lb\xb^0,\yb^*\ra)}{d_{max}} + 16m\mathbf{V}(\xb^0,\xb^*) + 2\|\yb^*\|^2+ \tfrac{4\mathcal{D}}{md_{max}}} \Big\}\\
&= {\cal O}\left\{\tfrac{m}{N+m}\left[\|\Lb\xb^0\|+ d_{max}\|\yb^*\| + d_{max}\sqrt{\tfrac{F(\xb^0)-F(\xb^*)+\la \Lb\xb^0,\yb^*\ra}{d_{max}} + m\mathbf{V}(\xb^0,\xb^*)+ \tfrac{\mathcal{D}}{md_{max}}}\right]\right\}.
\end{align*}
Furthermore, by \eqref{bnd_Qv2} we have
\begin{align*}
  \bbe\{g(\vb,\bar\zb^N)\}\le \tfrac{m}{N+m}\left[F(\xb^0)-F(\xb^*) +  8md_{max}\mathbf{V}(\xb^0,\xb^*) + \tfrac{2\mathcal{D}}{m}\right].
\end{align*}
The results in \eqref{bnd_aasdcs} 
immediately follow from applying Proposition~\ref{prop:approx} to the above two inequalities.
\hfill $\square$

In the following proposition, we provide the main convergence property of the AA-SDCS method stated in Algorithm~\ref{alg:AASDCS} when the objective functions $f_i, i=1,\dots,m$, are strongly convex.
\begin{proposition}\label{prop_bndQ_aasdcs_s}
Let the iterates $(\underline \xb^k, \xb^{k})$ and $\hat\yb^k$, $k=1,\dots,N$, be generated by Algorithm~\ref{alg:AASDCS} and be defined as in \eqref{def_yhat}, respectively, and let $\bar \zb^N:=(\tsum_{k=0}^N\theta_k\underline \xb^k,\tsum_{k=0}^N\hat\theta_k\hat\yb^k)$. 
Assume that the objective $f_i, i = 1,\dots,m$, are strongly convex functions, i.e., $\mu>0, \ L,M\ge 0$ in \eqref{eqn:nonsmooth}. 
Let the parameters $\{\alpha_k\}$, $\{\tau_k\}$, and $\{\eta_k\}$ in Algorithm~\ref{alg:AASDCS} satisfy \eqref{theta_tau}, \eqref{alpha_htheta} - \eqref{m_d_eta_tau}, 
\begin{align}
\hat\theta_k\left(\tfrac{\mathcal{C} + L}{T_k(T_k+1)}+\eta_k\right)&\le \hat\theta_{k-1}\left(\tfrac{\mathcal C +L}{T_{k-1}(T_{k-1}+1)}+\eta_{k-1}+\mu\right), \ k=2,\dots,N, \label{theta_Tk_eta_s}
\end{align}
where $\{\hat \theta_k\}$ is some given weight sequence.
Let the parameters $\{\lambda_t\}$ and $\{\beta_t\}$ in the \inner~procedure of Algorithm~\ref{alg:AASDCS} be set to \eqref{para_sgd1}.
Then, for any $\zb:=(\xb,\yb)\in X^m\times\bbr^{md}$, we have
\beq\label{bnd_Q_aasdcs_s}
\bbe\{Q(\bar\zb^N;\zb)\}\le \hat\theta_0(F(\xb^0)-F(\xb)) + m\hat\theta_1\left(\tfrac{4(\mathcal C+L)}{T_1(T_1+1)}+\eta_1\right)\mathbf{V}(\xb^0,\xb)+ \bbe\{\langle \vb,\yb\rangle\}+ \tsum_{k=1}^{N}\tfrac{8m(M^2+\sigma^2)\hat\theta_k}{(\eta_k+\mu)(T_k+1)},
\eeq
where $\bbe$ represents the taking expectation over all random variables, $Q$ and $\vb$ are defined in \eqref{eqn:gap} and \eqref{def_resi_aasdcs} 
respectively. Furthermore, for any saddle point $(\xb^*,\yb^*)$ of \eqref{eqn:saddle}, we have
\begin{align}\label{bnd_xyn_aasdcs_s}
\tfrac{\hat\theta_N}{4}&\left(1-\tfrac{2\|\Lb\|^2}{m\eta_N\tau_N}\right)\max\left\{\eta_N\bbe\|\hat{\underline \xb}^N-\xb^{N-1}\|^2,2m\tau_N\bbe\|\yb^*-\yb^N\|^2\right\}\nn\\
&\le \hat\theta_0(F(\xb^0)-F(\xb^*) + \langle\Lb\xb^0,\yb^*\rangle) + m\hat\theta_1\left(\tfrac{4(\mathcal C+L)}{T_1(T_1+1)}+\eta_1\right)\mathbf{V}(\xb^0,\xb^*)\nn\\
&\quad + \tfrac{m\hat\theta_1\tau_1}{2}\|\yb^*\|^2 + \tsum_{k=1}^N\tfrac{8m(M^2+\sigma^2)\hat\theta_k}{(T_k+1)(\eta_k+\mu)}.
\end{align}
\end{proposition}

\begin{proof}
Since $f_i$'s are strongly convex function, we have $\mu > 0$ and $L,M\ge0$ (cf. \eqref{eqn:nonsmooth}). Observe that $\lambda_t$ and $\beta_t$ defined in \eqref{para_sgd1} satisfy conditions \eqref{lam_1}-\eqref{beta_Lambda}. Therefore, following similar procedure in Proposition~\ref{prop_bndQ_aasdcs}, in view of Proposition~\ref{prop:inner}, and
the definition of $\hat \xb^k_{+}$ and $\hat{\underline \xb}^k$ in \eqref{def_xhatu}, we can obtain 
  \begin{align*}
  \bbe_{[\xi]}\Phi^k(\hat{\underline \xb}^k) - \Phi^k(\xb) 
&\le \Lambda_{T_k} \beta_1 \mathbf{V}(\xb^{k-1},\xb) - (\Lambda_{T_k}\beta_1+ \mu +\eta_k)\bbe_{\xi}\mathbf{V}(\xb_+^k,\xb) + \Lambda_{T_k}\tsum_{t=1}^{T_k}\tfrac{2m(M^2+\sigma^2)\lambda_t^2}{(\mu + \eta_k + \beta_t - (\mathcal{C}+L)\lambda_t^2)\Lambda_t},
  \end{align*}
  where $\Phi^k(\xb)= \la \Lb\xb, \tilde\yb^k\ra + F(\xb) + \eta_k\mathbf{V}(\xb^{k-1},\xb)$. By plugging into the above relation the values of $\lambda_t$ and $\beta_t$ in \eqref{para_sgd1}, together with the definition of $\Phi^k(\xb)$ and rearranging the terms, we have $\forall \xb\in X^m$
  \begin{align}\label{condi_sgd_s}
  \bbe_{[\xi]}&\left\{\la \Lb(\hat {\underline \xb}^k-\xb), \tilde \yb^k\ra + F(\hat {\underline \xb}^k)- F(\xb)\right\}
  \le \left(\tfrac{4(\mathcal{C}+L)}{T_k(T_k+1)}+\eta_k\right)\bbe_{[\xi]}\{\mathbf{V}(\xb^{k-1},\xb)\}\nn\\
  &\quad - \left(\tfrac{4(\mathcal{C}+L)}{T_k(T_k+1)}+\eta_k + \mu\right)\bbe_{[\xi]}\{\mathbf{V}(\hat\xb^k_{+},\xb)\}
  -\eta_k\bbe_{[\xi]}\{\mathbf{V}(\xb^{k-1},\hat {\underline\xb}^k)\}+\tfrac{8m(M^2+\sigma^2)}{(T_k+1)(\eta_k+\mu)}.
  \end{align}
Observe that if we replace \eqref{condi_sgd} by \eqref{condi_sgd_s} in Proposition~\ref{prop_bndQ_aasdcs}, we can conclude the following result similar to \eqref{bnd_Q_aasdcs_con}
\begin{align}\label{bnd_Q_aasdcs_con_s}
\bbe\{Q(\bar\zb^N;\zb)\}&\le \hat\theta_0Q_0(\xb,\yb) + \tsum_{k=1}^N\tfrac{8m(M^2+\sigma^2)\hat\theta_k}{(T_k+1)(\eta_k+\mu)} + \bbe\left\{\tsum_{k=1}^N\hat\theta_k\bar \Delta_k\right\},
\end{align}
where $\bbe$ represents taking the expectation over all random variables and 
\begin{align}\label{def_bDelta}
\bar\Delta_k:= &\langle \Lb(\hat{\underline \xb}^k-\tilde\xb^k),\yb-\tilde\yb^{k}\rangle 
+ m\left[\left(\tfrac{4(\mathcal{C}+L)}{T_k(T_k+1)}+\eta_k\right)\mathbf{V}(\xb^{k-1},\xb)
-\left(\tfrac{4(\mathcal{C}+L)}{T_k(T_k+1)}+\eta_k+\mu\right)\mathbf{V}(\xb^k,\xb)\right]\nn\\
&\quad -\eta_k\mathbf{V}(\xb^{k-1},\hat{\underline \xb}^k)
 + \tfrac{m\tau_k}{2}[\|\yb-\yb^{k-1}\|^2-\|\yb-\yb^k\|^2-\|y_{i_k}^{k-1}-y_{i_k}^k\|^2].
\end{align}
Since $\bar \Delta_k$ defined above shares a similar structure with $\tilde \Delta_k$ in \eqref{def_tDelta}, we can follow a similar procedure as in Proposition~\ref{prop_bndQ_aasdcs} to obtain a bound for $\bbe\{Q(\bar\zb^N,\zb)\}$. Note that the only difference between \eqref{def_bDelta} and \eqref{def_tDelta} exists in the coefficient of the term $\mathbf{V}(\xb^{k-1},\xb)$ and $\mathbf{V}(\xb^k,\xb)$. Hence, by using condition \eqref{theta_Tk_eta_s} in place of \eqref{theta_Tk_eta}, we obtain 
\begin{align*}
\bbe\left\{Q(\bar \zb^N,\zb)\right\}
&\le \hat\theta_0(F(\xb^0)-F(\xb)) + m\hat\theta_1\left(\tfrac{4(\mathcal C+L)}{T_1(T_1+1)}+\eta_1\right)\mathbf{V}(\xb^0,\xb)+ \tsum_{k=1}^N\tfrac{8m(M^2+\sigma^2)\hat\theta_k}{(T_k+1)(\eta_k+\mu)}
\nn\\
&\quad + \bbe\left\{\langle \hat\theta_0 \Lb\xb^0+ \hat\theta_N\Lb(\hat{\underline\xb}^N-\xb^{N-1})+m\hat\theta_1\tau_1\yb^N,\yb\rangle\right\}.
\end{align*}
Our result in \eqref{bnd_Q_aasdcs_s} immediately follows. Following the same procedure as we obtain \eqref{bnd_huxN}, for any saddle point $\zb^*=(\xb^*,\yb^*)$ of \eqref{eqn:saddle}, we have
\begin{align*}
\tfrac{\hat\theta_N\eta_N}{4}\bbe\|\hat{\underline \xb}^{N-1}-\xb^N\|^2
&\le \hat\theta_0(F(\xb^0)-F(\xb^*) + \langle\Lb\xb^0,\yb^*\rangle) + m\hat\theta_1\left(\tfrac{4(\mathcal C+L)}{T_1(T_1+1)}+\eta_1\right)\mathbf{V}(\xb^0,\xb^*)\nn\\ 
&\quad + \tfrac{m\hat\theta_1\tau_1}{2}\|\yb^*\|^2+ \bbe\left\{\tfrac{\hat\theta_N\|\Lb\|^2}{2m\tau_N}\|\hat{\underline \xb}^{N-1}-\xb^{N}\|^2\right\} + \tsum_{k=1}^N\tfrac{8m(M^2+\sigma^2)\hat\theta_k}{(T_k+1)(\eta_k+\mu)},\\
\tfrac{m\hat\theta_N\tau_N}{2}\bbe\|\yb^*-\yb^N\|^2 
&\le \hat\theta_0(F(\xb^0)-F(\xb^*) + \langle\Lb\xb^0,\yb^*\rangle) + m\hat\theta_1\left(\tfrac{4(\mathcal C+L)}{T_1(T_1+1)}+\eta_1\right)\mathbf{V}(\xb^0,\xb^*)\\
&\quad + \tfrac{m\hat\theta_1\tau_1}{2}\|\yb^*\|^2
 + \bbe\{\tfrac{m\hat\theta_N\|\Lb\|^2}{\eta_N}\|\yb^*-\yb^N\|^2\} + \tsum_{k=1}^N\tfrac{8m(M^2+\sigma^2)\hat\theta_k}{(T_k+1)(\eta_k+\mu)},
\end{align*}
from which the result in \eqref{bnd_xyn_aasdcs_s} follows.
\end{proof}

{\bf Proof of Theorem~\ref{main_aasdcs_s}} Let us set 
\beq\label{def_htheta_s}
\hat\theta_k = 
\begin{cases}
\tfrac{6m^2}{6m^2+ N(N+6m + 1)}, \ &k = 0,\\
\tfrac{2(k+3m)}{6m^2+N(N+6m+1)}, \ &k = 1,\dots,N.
\end{cases}
\eeq 
Observe from \eqref{para_aasdcs_s} that
\[
\eta_k = \tfrac{(k+3m-1)\mu}{2} - \tfrac{\mathcal C + L}{T_k(T_k+1)}\ge \tfrac{(k+3m-1)\mu}{2} - \tfrac{(k+3m-3)\mu}{4} = \tfrac{(k+3m + 1)\mu}{4}.
\]
Therefore, it is easy to check that parameter settings \eqref{para_sgd1} and \eqref{para_aasdcs_s} satisfy conditions \eqref{lam_1} - \eqref{beta_Lambda}, \eqref{theta_tau}, \eqref{alpha_htheta} - \eqref{m_d_eta_tau}, and \eqref{theta_Tk_eta_s}. Also by \eqref{def_xweight}, we have
\[
\theta_k=
  \begin{cases}
  \tfrac{2(k+ 2m+1)}{6m^2 + N(N+6m+1)}, \ &k=1,\dots,N-1,\\
  \tfrac{2m(N+3m)}{6m^2+N(N+6m+1)}, \ &k=N,
  \end{cases}
\]
which implies that $\bar\xb^N=\tfrac{2}{6m^2+N(N+6m+1)}(\tsum_{k=0}^{N-1}(k+2m+1)\underline \xb^k+ m(N+3m)\underline \xb^N)$. By plugging the parameter setting in \eqref{bnd_Q_aasdcs_s}, we have
\begin{align}\label{bnd_Qv3}
\bbe\{Q(\bar\zb^N;\xb^*,\yb)\}&\le \tfrac{6m^2}{6m^2+N(N+6m+1)}\left[F(\xb^0)-F(\xb^*) + \tfrac{(3m+1)\mu}{2}\mathbf{V}(\xb^0,\xb^*) + \tfrac{\mathcal{D}\mu}{6m^2}\right]
+ \bbe\{\langle \vb,\yb\rangle\}.
\end{align} 
Observe that from \eqref{def_resi_aasdcs} and \eqref{para_aasdcs_s}
\begin{align*}
\bbe\{\|\vb\|\}&\le \tfrac{2m^2}{6m^2+N(N+6m+1)}\bbe\left[3\|\Lb\xb^0\| + \tfrac{(N+3m)\|\Lb\|}{m^2}\|\hat{\underline \xb}^N-\xb^{N-1}\|+\tfrac{32d_{max}^2}{\mu}(\|\yb^*-\yb^N\| + \|\yb^*\|)\right].
\end{align*}
In view of \eqref{bnd_xyn_aasdcs_s} and \eqref{para_aasdcs_s}, we have
\begin{align*}
\bbe\{\|\hat{\underline \xb}^N-\xb^{N-1}\|^2\}
&\le\tfrac{8}{\hat\theta_N\eta_N}\tfrac{6m^2}{6m^2+N(N+6m+1)}\left[F(\xb^0)-F(\xb^*) + \la \Lb\xb^0,\yb^*\ra + \tfrac{(3m+1)\mu}{2}\mathbf{V}(\xb^0,\xb^*)\right.\\
&\quad \left. + \tfrac{32d_{max}^2}{3\mu}\|\yb^*\|^2 + \tfrac{\mathcal{D}\mu}{6m^2}\right]\\
&\le \tfrac{96m^2}{(N+3m)(N+3m+1)}\left[\tfrac{F(\xb^0)-F(\xb^*) + \la \Lb\xb^0,\yb^*\ra}{\mu} + \tfrac{(3m+1)}{2}\mathbf{V}(\xb^0,\xb^*)\right.\\
&\quad \left. + \tfrac{32d_{max}^2}{3\mu^2}\|\yb^*\|^2 + \tfrac{\mathcal{D}}{6m^2}\right],\\
\bbe\{\|\yb^*-\yb^N\|^2\}&\le\tfrac{3\mu^2}{4d_{max}^2}\left[\tfrac{F(\xb^0)-F(\xb^*) + \la \Lb\xb^0,\yb^*\ra}{\mu} + \tfrac{(3m+1)}{2}\mathbf{V}(\xb^0,\xb^*)+ \tfrac{32d_{max}^2}{3\mu^2}\|\yb^*\|^2 + \tfrac{\mathcal{D}}{6m^2} \right].
\end{align*}
Hence, in view of the above three inequalities and Jensen's inequality, we obtain
\begin{align*}
\bbe\{\|\vb\|\}&\le \tfrac{2m^2}{6m^2+N(N+6m+1)}\left\{3\|\Lb\xb^0\|+\tfrac{32d_{max}^2}{\mu}\|\yb^*\|\right.\\
&\quad \left. + 24d_{max}\sqrt{\tfrac{3(F(\xb^0)-F(\xb^*) + \la \Lb\xb^0,\yb^*\ra)}{\mu} + \tfrac{3(3m+1)}{2}\mathbf{V}(\xb^0,\xb^*) + \tfrac{32d_{max}^2}{\mu^2}\|\yb^*\|^2 + \tfrac{\mathcal{D}}{2m^2}}\right\}\\
&= {\cal O}\left\{\tfrac{m^2}{m^2+N^2}\left[\|\Lb\xb^0\|+\tfrac{d_{max}^2}{\mu}\|\yb^*\| + d_{max}\sqrt{\tfrac{(F(\xb^0)-F(\xb^*) + \la \Lb\xb^0,\yb^*\ra)}{\mu} + m\mathbf{V}(\xb^0,\xb^*) + \tfrac{\mathcal{D}}{m^2}}\right]\right\}.
\end{align*}

Furthermore, by \eqref{bnd_Qv3} we have
\begin{align*}
  \bbe\{g(\vb,\bar\zb^N)\}\le \tfrac{6m^2}{6m^2+N(N+6m+1)}\left[F(\xb^0)-F(\xb^*) + \tfrac{(3m+1)\mu}{2}\mathbf{V}(\xb^0,\xb^*) + \tfrac{\mathcal{D}\mu}{6m^2}\right].
\end{align*}
The results in \eqref{bnd_aasdcs_s} immediately follow from applying Proposition~\ref{prop:approx} to the above two inequalities.
\hfill $\square$

\end{document}